\documentclass[11pt]{article}
\input{style}
\usepackage{enumitem}
\usepackage{varwidth}

\begin{document}
\title{An optimal transport problem with backward martingale
  constraints motivated by insider trading.}

\author{Dmitry Kramkov\footnote{Carnegie Mellon University, Department
    of Mathematical Sciences, 5000 Forbes Avenue, Pittsburgh, PA,
    15213-3890, USA. {\it Email: }kramkov@cmu.edu } \and Yan
  Xu\footnote{Carnegie Mellon University, Department of Mathematical
    Sciences, 5000 Forbes Avenue, Pittsburgh, PA, 15213-3890,
    USA. {\it Email: }yanx1@andrew.cmu.edu}}

\date{\today}

\maketitle

\begin{abstract}
  We study a single-period optimal transport problem on $\R^2$ with a
  covariance-type cost function $c(x,y) = (x_1-y_1)(x_2-y_2)$ and a
  backward martingale constraint.  We show that a transport plan
  $\gamma$ is optimal if and only if there is a maximal monotone set
  $G$ that supports the $x$-marginal of $\gamma$ and such that
  $c(x,y) = \min_{z\in G}c(z,y)$ for every $(x,y)\in \supp\gamma$.  We
  obtain sharp regularity conditions for the uniqueness of an optimal
  plan and for its representation in terms of a map.  Our study is
  motivated by a variant of the classical Kyle model of insider
  trading from \citet{RochVila:94}.
\end{abstract}

\begin{description}
\item[Keywords:] martingale optimal transport, Kyle equilibrium.
\item[AMS Subject Classification (2010):] 60G42, 91B24, 91B52.
\end{description}

\section{Introduction}
\label{sec:introduction}

Let $(\Omega, \cF, \mathbb{P})$ be a probability space and
$Y = (Y_1, Y_2)$ be a $2$-dimensional random variable with finite
second moment: $Y\in \mathcal{L}^2(\Omega, \cF, \mathbb{P})$. Our goal
is to
\begin{equation}
  \label{eq:1}
  \text{minimize} \quad \EP{c(X,Y)} 
  \quad \text{over} \quad X\in \mathcal{X}(Y),
\end{equation}
for the cost function $c(x,y) = (x_1 - y_1)(x_2 - y_2)$,
$x,y\in \R^2$, and the domain $\mathcal{X}(Y)$ that consists of
$Y$-measurable random variables $X=(X_1,X_2)$ such that $(X,Y)$ is a
martingale: $\cEP{Y}{X} = X$. A relaxation of the $Y$-measurability
constraint on $X$ leads to the optimal transport problem:
\begin{equation}
  \label{eq:2}
  \text{minimize}\int c(x,y) d\gamma  \quad \text{over} \quad \gamma
  \in \Gamma(\nu),
\end{equation}
where $\nu=\law{Y}$ and $\Gamma(\nu)$ is the family of probability
measures $\gamma = \gamma(dx,dy)$ on $\R^2 \times \R^2$ that have
$\nu$ as their $y$-marginal: $\gamma(\R^2,dy) = \nu(dy)$, and make a
martingale out of the canonical process: $\gamma(y|x) = x$.  In view
of the martingale constraint, problem~\eqref{eq:2} admits an
equivalent formulation:
\begin{align*}
  \text{maximize} \quad \int x_1x_2 d\gamma  \quad \text{over} \quad
  \gamma \in \Gamma(\nu),
\end{align*}
and thus, has a natural connection to the classical
Fr\'echet-Hoeffding inequality and the Wasserstein $2$-distance.

Problem~\eqref{eq:2} exhibits a \emph{backward} structure in the sense
that the initial marginal $\mu(dx) = \gamma(dx,\R^2)$ is part of the
solution. In this regard, it differs from the ``standard''
single-period martingale transport problem 
in \citet{BeigJuil:16}, \citet{BeigNutzTouz:17}, \citet{HenrTouz:16},
and \citet{GhousKimLim:19}, among others, where both the initial and
terminal marginals are fixed. We point out that for our cost function
$c(x,y)=(x_1-y_1)(x_2-y_2)$, the standard problem is trivial, as every
martingale measure $\gamma=\gamma(dx,dy)$ with given marginals
$\mu=\mu(dx)$ and $\nu = \nu(dy)$ produces the same average cost:
\begin{displaymath}
  \int c(x,y) d\gamma = \int y_1y_2 d\nu - \int x_1x_2 d\mu. 
\end{displaymath}

Our work is motivated by the classical \citet{Kyle:85} equilibrium
with insider from financial economics. More precisely, we consider the
model from~\citet{RochVila:94}, where the insider observes \emph{both}
the terminal value $V$ of the risky asset and the order flow $U$ of
the noise traders; see Section~\ref{sec:equil-with-insid} for details.
Setting $Y=(U,V)$ we establish in Theorem~\ref{th:6} the equivalence
between the existence of equilibrium and that of an optimal \emph{map}
$X$ for~\eqref{eq:1} such that $\gamma=\Law{(X,Y)}$ is an optimal
\emph{plan} for~\eqref{eq:2}. Moreover, the components of $X=(R,S)$
are naturally identified as equilibrium's total order $R$ and price
$S$. To the best of our knowledge, the connection between the Kyle
equilibrium and a martingale optimal transport is new.

The main results of the paper are Theorems~\ref{th:1}
and~\ref{th:5}. In Theorem~\ref{th:1} we prove the existence of an
optimal {plan} for~\eqref{eq:2} and characterize its support.  We show
that $\gamma\in \Gamma(\nu)$ is optimal if and only if there is a
\emph{maximal monotone set} $G$ in $\R^2$ that supports the
$x$-marginal of $\gamma$ and such that
\begin{equation}
  \label{eq:3}
  c(x,y) = \phi_G(y) \set \inf_{z\in G}c(z,y), \quad (x,y) \in \supp
  \gamma. 
\end{equation}
Geometrically, the support of ${\gamma}$ has the \emph{hyperbolic
  tangent property}: it connects $y\not \in G$ to those $x\in G$, that
are touched by the hyperbola
\begin{displaymath}
  H = \descr{z\in \R^2}{c(y,z) = \phi_G(y)} = \descr{z\in \R^2}{z_2 =
    y_2 + \frac{\phi_G(y)}{z_1-y_1}};
\end{displaymath}
see Figure~\ref{fig:1}.  Surprisingly, as a consequence
of~\eqref{eq:3}, the optimal plan $\gamma$ possesses properties of
solutions to classical unconstrained problems. By
Corollary~\ref{cor:1}, the $x$-marginal of $\gamma$ is a
Fr\'echet-Hoeffding coupling between its first and second coordinates,
while, by Corollary~\ref{cor:2}, $\gamma$ is a classical optimal
coupling between its $x$- and $y$-marginals.

In Theorem~\ref{th:2} we show that the set $G$ from~\eqref{eq:3} is a
solution of the \emph{dual} problem:
\begin{equation}
  \label{eq:4}
  \text{maximize} \quad \int \phi_G d\nu \quad \text{over}\quad G\in
  \mathfrak{M},   
\end{equation}
where $\mathfrak{M}$ is the family of all {maximal monotone} sets in
$\R^2$, and that primal and dual problems~\eqref{eq:2}
and~\eqref{eq:4} have identical values. The dual problem appears in
\citep[Eq.~(2.3)]{RochVila:94}, where $G$ stands for the graph of a
pricing rule. When $\nu=\law{Y}$ has a Gaussian or, more generally,
elliptically contoured distribution, $G$ becomes a line with strictly
positive slope; see Example~\ref{ex:1}.

In Theorem~\ref{th:3}, we show that optimal {map} and {plan}
problems~\eqref{eq:1} and~\eqref{eq:2} have identical values, provided
that $\nu$ is atomless. The result is similar to that of
\citet{Prat:07} for the classical unconstrained case.  The existence
of an optimal map $X$ for~\eqref{eq:1} that induces an optimal plan
$\gamma=\law{X,Y}$ for~\eqref{eq:2} is obtained in Theorem~\ref{th:4}
under the condition that $\nu$ gives zero mass to the graphs of
\emph{strictly decreasing} Lipschitz functions. This assumption is
weaker than the standard regularity condition of the Brenier theorem,
see~\citep[Theorem~1.26]{AmbrGigli:13}, that requires $\nu$ to assign
zero mass to rotations of the graphs of Lipschitz functions.  Our
second main result, Theorem~\ref{th:5}, establishes the uniqueness of
solutions to~\eqref{eq:1} and~\eqref{eq:2} if, in addition, the
(one-dimensional) distribution functions of $Y_1$ and $Y_2$ are
continuous.  Examples~\ref{ex:2} and~\ref{ex:3} show that the
conditions of Theorems~\ref{th:3} and~\ref{th:4} are sharp.

Being applied to the model of \citet{RochVila:94}, Theorems~\ref{th:4}
and~\ref{th:5} yield sufficient conditions for the existence and
uniqueness of equilibria, which are stated in
Theorem~\ref{th:7}. These assumptions generalize those
in~\citet{RochVila:94}, where $Y=(U,V)$ is required to have a
continuous compactly supported density in $\R^2$.  \citet{RochVila:94}
work with dual problem~\eqref{eq:4} and rely on the properties of the
space of closed graph correspondences endowed with the Hausdorff
topology.

Finally, Appendix~\ref{sec:density-cw_2rd} contains a density result
for the Wasserstein spaces, for which we could not find a ready
reference, while Appendix~\ref{sec:dual-space-functions} collects the
properties of the function $\phi_G$ from~\eqref{eq:3}.

\section{A backward martingale optimal transport problem}
\label{sec:optim-transp-probl}

We denote by $\mathcal{P}_2(\R^d)$ the family of Borel probability
measures with finite second moments and by $\mathcal{B}(\R^d)$ the
Borel $\sigma$-algebra on $\R^d$.  For a Borel probability measure
$\mu$ on $\R^d$, a $\mu$-integrable $m$-dimensional Borel function
$f=(f_1,\dots,f_m)$, and an $n$-dimensional Borel function
$g = (g_1,\dots,g_n)$, the notation $\mu(f|g)$ stands for the
$m$-dimensional vector of conditional expectations of $f_i$ given $g$
under $\mu$:
\begin{displaymath}
  \mu(f|g) = (\mu(f_1|g_1,\dots,g_n),\dots,\mu(f_m|g_1,\dots,g_n)).
\end{displaymath}
Similarly, $\int f d\mu = (\int f_1 d\mu, \dots, \int f_m d\mu)$.  We
write a point in $\R^4 = \R^2\times \R^2$ as $(x,y)$, where
$x=(x_1,x_2)$ and $y=(y_1,y_2)$ belong to $\R^2$, and think about $x$
and $y$ as the initial and terminal values of the canonical
two-dimensional process.

Let $\nu = \nu(dy)\in \mathcal{P}_2(\R^{2})$. We denote by
$\Gamma(\nu)$ the family of probability measures
$\gamma = \gamma(dx,dy) \in \mathcal{P}_2(\R^{2}\times \R^2)$ that
have $\nu$ as their $y$-marginal and make a martingale out of the
canonical process:
\begin{displaymath}
  \Gamma(\nu) \set \descr{\gamma \in
    \mathcal{P}_2(\R^2\times \R^2)}{\gamma(\R^2,dy) = \nu(dy) \text{
      and }
    \gamma(y|x)=x}. 
\end{displaymath}
Our goal is to
\begin{equation}
  \label{eq:5}
  \text{minimize} \quad \int c(x,y)d\gamma \quad \text{over} \quad
  \gamma\in \Gamma(\nu) 
\end{equation}
for the \emph{covariance-type} cost function
\begin{displaymath}
  c(x,y) =  (x_1 - y_1)(x_2-y_2), \quad x,y\in \R^2.  
\end{displaymath} 
Problem~\eqref{eq:5} belongs to the class of optimal transport
problems with \emph{backward} martingale constraints, in the sense
that the initial $x$-marginal is part of the solution. As we shall see
in Section~\ref{sec:equil-with-insid}, such problem naturally arises
in the study of the Kyle-type equilibrium with insider.

\begin{Remark}
  \label{rem:1}
  Problem~\eqref{eq:5} admits several equivalent formulations. For
  instance, it has same solutions as the one, where we
  \begin{equation}
    \label{eq:6}
    \text{maximize} \int x_1 x_2 d\gamma \quad\text{over}\quad \gamma\in
    \Gamma(\nu). 
  \end{equation} 
  The justification comes from the identity
  \begin{align*}
    \int c(x,y) d\gamma &= \int (x_1 - y_1)(x_2-y_2) d\gamma = \int
                          (y_1y_2 - x_1x_2)d\gamma  \\
                        &= \int y_1 y_2 d\nu - \int x_1 x_2 d\gamma,
  \end{align*}
  where the second equality holds by the martingale property of
  $\gamma \in \Gamma(\nu)$.
\end{Remark}

For a Borel probability measure $\gamma$ on $\R^d$ we denote by
$\supp{\gamma}$ its \emph{support}, that is, the smallest closed set
with full measure.  We recall that a set $G\subset \R^{2}$ is
\emph{monotone} if
\begin{displaymath}
  c(r,s) = (r_1 - s_1)(r_2-s_2) \geq 0,  \quad r,s\in G. 
\end{displaymath}
A monotone set $G$ is \emph{maximal} if it is not a proper (or strict)
subset of a monotone set. We denote by $\mathfrak{M}$ the family of
maximal monotone sets in $\R^2$. It is well-known that
$G\in \mathfrak{M}$ if and only if $G$ is the graph of the
subdifferential of a proper closed convex function on $\R$.

For $G\in \mathfrak{M}$ we define a function
\begin{displaymath}
  \phi_G(y) \set \inf_{x\in G} c(x,y) = \inf_{x\in G}
  (x_1-y_1)(x_2-y_2), \quad y\in \R^2.  
\end{displaymath}
Such functions $\phi_G$ will play a key role in our study. Their
properties are collected in
Appendix~\ref{sec:dual-space-functions}. In particular,
Lemma~\ref{lem:16} states that $\phi_G$ takes values in $[-\infty,0]$
and $G = \descr{x\in \R^2}{\phi_G(x)=0}$.

The main results of the paper are Theorems~\ref{th:1}
and~\ref{th:5}. Theorem~\ref{th:1} establishes the existence of an
optimal plan $\gamma$ for~\eqref{eq:5} and shows the structure of its
support. Theorem~\ref{th:5} contains a uniqueness result.

\begin{Theorem}
  \label{th:1}
  Let $\nu \in \mathcal{P}_2(\R^2)$. An optimal plan for~\eqref{eq:5}
  exists. For a probability measure $\gamma \in \Gamma(\nu)$ the
  following conditions are equivalent:
  \begin{enumerate}[label = {\rm (\alph{*})}, ref={\rm (\alph{*})}]
  \item \label{item:1} $\gamma$ is an optimal plan for~\eqref{eq:5}.
  \item \label{item:2} If points $(x^0,y^0)$ and $(x^1, y^1)$ belong
    to $\supp{\gamma}$, then
    \begin{equation}
      \label{eq:7}
      (1-t) c(x^0,y^0) + t c(x^1,y^1) \leq t (1-t) c(y^0, y^1), \quad  
      t\in [0,1].   
    \end{equation}
  \item \label{item:3} There is $G\in \mathfrak{M}$ such that
    \begin{equation}
      \label{eq:8}
      c(x,y) \leq  \phi_G(y), \quad (x,y)\in \supp{\gamma}.
    \end{equation}
  \end{enumerate}
  Moreover, if $G$ is a maximal monotone set satisfying~\eqref{eq:8}
  and $\mu$ is the $x$-marginal of $\gamma$, then $G$ contains
  $\supp{\mu}$ and
  \begin{displaymath}
    c(x,y) = \phi_G(y) = \min_{z\in G} c(z,y), \quad (x,y)\in \supp{\gamma}.
  \end{displaymath}
\end{Theorem}

\begin{figure}
  \centering
  \begin{tikzpicture}[scale = 4/10]
   
    \draw[dashed, ultra thin] (-3,-1) node [above left] {$x^1$} --(2,
    -6) node [right] {$y^1$};

    \draw[dashed, ultra thin] (3,1) node [below right] {$x^0$} --(-2,
    6) node [left] {$y^0$};

    \filldraw [black] (-3,-1) circle [radius=2pt] (2,-6) circle
    [radius=2pt] (3,1) circle [radius=2pt] (-2,6) circle [radius=2pt];

    \draw[thick,domain=-10:-0.09,smooth,variable=\x,black] plot
    ({\x},{-25/(\x-2) - 6}); \draw (-1,2.6667) node [left] {$H^1$};

    \draw[thick,domain=0.09:10,smooth,variable=\x,black] plot
    ({\x},{-25/(\x+2) + 6}); \draw (1,-2.6667) node [right] {$H^0$};

    \draw[very thick,domain=-5.45:5.45,smooth,variable=\x,black] plot
    ({\x},{(\x/3)^3}); \draw (5, 125/27) node [above left] {$G$};

  \end{tikzpicture}
  \caption{Hyperbolic non-crossing and tangent properties of the
    support of optimal plan.}
  \label{fig:1}
\end{figure}
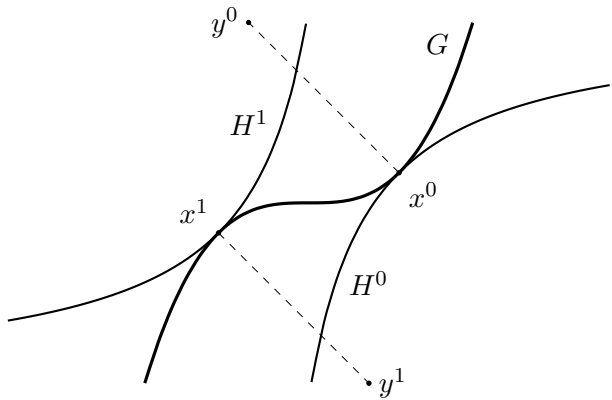

Figure~\ref{fig:1} illustrates the properties of the support of an
optimal plan $\gamma$ stated in Theorem~\ref{th:1}. Let $(x^0,y^0)$
and $(x^1,y^1)$ belong to $\supp{\gamma}$ and be such that
$x^0\not=x^1$ and the points $y^0$ and $y^1$ lie, respectively,
strictly above and strictly below the maximal monotone set $G$ from
item~\ref{item:3}.  As Lemma~\ref{lem:5} shows, item~\ref{item:2}
means that the hyperbolas
\begin{align*}
  H^0 = \descr{z\in \R^2}{c(z,y^0) = c(x^0,y^0),\; z_1 > y^0_1}, \\
  H^1 = \descr{z\in \R^2}{c(z,y^1) = c(x^1,y^1),\; z_1 < y^1_1},
\end{align*}
do \emph{not cross}. The geometric interpretation of item~\ref{item:3}
is that these hyperbolas are \emph{tangent} to $G$.

Before proceeding with the proof of Theorem~\ref{th:1}, we establish
rather surprising connections between an optimal martingale plan
for~\eqref{eq:5} and the solutions of classical unconstrained optimal
transport problems.  If $\mu$ and $\nu$ are Borel probability measures
on $\R^d$, then $\Pi(\mu, \nu)$ denotes the family of all couplings of
$\mu$ and $\nu$, that is, the family of Borel probability measures
$\pi$ on $\R^d\times \R^d = \descr{(x,y)}{x,y\in \R^d}$ with
$x$-marginal $\mu$ and $y$-marginal $\nu$.  For
$\mu, \nu \in \mathcal{P}_2(\R^d)$, the Wasserstein 2-metric is given
by
\begin{equation}
  \label{eq:9}
  W_2(\mu,\nu)  \set \inf_{\pi \in \Pi(\mu,\nu)} \sqrt{\int
    \abs{x-y}^2 d\pi}. 
\end{equation}

\begin{Corollary}
  \label{cor:1}
  Let $\nu \in \mathcal{P}_2(\R^2)$, $\mu$ be the $x$-marginal of an
  optimal plan $\gamma$ for~\eqref{eq:5}, and $\mu_i$ be the
  $x_i$-marginal of $\mu$, $i=1,2$. Then $\mu$ is a solution of the
  optimal transport problem:
  \begin{equation}
    \label{eq:10}
    \text{maximize} \quad \int x_1 x_2 d\pi \quad
    \text{over} \quad \pi\in \Pi(\mu_1,\mu_2), 
  \end{equation}
  or, equivalently,
  \begin{displaymath}
    W_2(\mu_1,\mu_2) = \sqrt{\int \abs{x_1-x_2}^2 d\mu}. 
  \end{displaymath}
\end{Corollary}

\begin{proof}
  It is well-known that problems~\eqref{eq:9} and~\eqref{eq:10} have
  same solutions and that an element of $\Pi(\mu_1,\mu_2)$ is such a
  solution if and only if its support belongs to a cyclically monotone
  set. By Theorem~\ref{th:1}, there is a monotone set $G$ that
  contains the support of $\mu$. Since every monotone set in $\R^2$ is
  also cyclically monotone, the result follows.
\end{proof}

\begin{Remark}
  \label{rem:2}
  Let $G$ be a maximal monotone set from Theorem~\ref{th:1} and $P_1$
  be its projection on $x_1$-coordinate.  If $\mu_1$ is atomless, then
  the increasing function
  \begin{displaymath}
    f(x_1) = \inf\descr{x_2\in \R}{(x_1,x_2)\in G}, \quad x_1\in P_1, 
  \end{displaymath}
  taking values in $\R\cup \braces{-\infty}$, defines an optimal map
  solution to~\eqref{eq:10}:
  \begin{displaymath}
    \mu(B) = \mu_1\descr{t\in \R}{(t,f(t))\in B}, \quad B\in
    \mathcal{B}(\R^2). 
  \end{displaymath}
  The function $f$ is a \emph{pricing rule} in a version of the Kyle
  equilibrium with insider studied in
  Section~\ref{sec:equil-with-insid}.
\end{Remark}

\begin{Corollary}
  \label{cor:2}
  Let $\nu \in \mathcal{P}_2(\R^2)$, $\gamma$ be an optimal plan
  for~\eqref{eq:5}, and $\mu$ be the $x$-marginal of $\gamma$. Then
  $\gamma$ is a solution of the optimal transport problem:
  \begin{equation}
    \label{eq:11}
    \text{minimize} \quad \int c(x,y)d\pi \quad
    \text{over} \quad \pi\in \Pi(\mu,\nu).   
  \end{equation}
\end{Corollary}

\begin{proof}
  By Theorem~\ref{th:1}, there is $G\in \mathfrak{M}$ such that
  \begin{displaymath}
    \int \phi d\nu = \int c(x,y) d\gamma, 
  \end{displaymath}
  where $\phi = \phi_G$.  Lemma~\ref{lem:16} shows that the
  $c$-conjugate function
  \begin{displaymath}
    \phi^c(x) \set \inf_{y\in \R^2} (c(x,y) - \phi(y)), \quad x\in
    \R^2,     
  \end{displaymath}
  takes values in $[-\infty,0]$ and
  $G = \descr{x\in \R^2}{\phi^c(x) = 0}$.  By Theorem~\ref{th:1},
  $\supp{\mu} \subset G$ and thus,
  \begin{displaymath}
    \int \phi^c d\mu = 0. 
  \end{displaymath}
  Since $\phi^c(x) + \phi(y) \leq c(x,y)$, we deduce that
  \begin{align*}
    \int c(x,y) d\gamma = \int \phi d\nu + \int \phi^c d\mu \leq \int
    c(x,y) d\pi, \quad \pi \in \Pi(\mu,\nu),
  \end{align*}
  and the optimality of $\gamma$ for~\eqref{eq:11} follows.
\end{proof}

\begin{Remark}
  \label{rem:3}
  We point out that the assertions of the corollaries are not
  sufficient for the optimality of $\gamma\in \Gamma(\nu)$. Indeed,
  let $\gamma$ be the simplest martingale measure, whose $x$-marginal
  $\mu$ is the Dirac measure concentrated at the mean
  $\int yd\nu \in \R^2$. In this case, the families $\Pi(\mu_1,\mu_2)$
  and $\Pi(\mu,\nu)$ are singletons and hence, $\mu$ and $\gamma$ are
  trivial solutions to~\eqref{eq:10} and~\eqref{eq:11}. An elementary
  analysis of~\eqref{eq:8} shows that such $\gamma$ is optimal
  for~\eqref{eq:5} if and only if the support of $\nu$ belongs to a
  line with negative or infinite slope.
\end{Remark}

The rest of the section is devoted to the proof of Theorem~\ref{th:1},
which we divide into lemmas. We start with the existence part and
recall some basic facts on the Wasserstein distance $W_2$; see
~\cite[Theorem~2.7 and Proposition~2.4]{AmbrGigli:13}. If $(\mu_n)$
and $\mu$ are in $\mathcal{P}_2(\R^d)$, then $W_2(\mu_n,\mu)\to 0$ if
and only if $\int f(x) d\mu_n \to \int f(x) d\mu$ for every continuous
function $f=f(x)$ on $\R^d$ with quadratic growth:
\begin{displaymath}
  \abs{f(x)} \leq K (1 + \abs{x}^2), \quad x\in \R^d, 
\end{displaymath}
where $K=K(f)>0$ is a constant.  A set $A\subset \mathcal{P}_2(\R^d)$
is \emph{pre-compact} under $W_2$ if and only if
\begin{displaymath}
  \sup_{\mu \in A} \int_{\abs{x}\geq K} \abs{x}^2 d\mu \to 0,
  \quad K\to \infty. 
\end{displaymath}

\begin{Lemma}
  \label{lem:1}
  The family $\Gamma(\nu)$ is a convex compact set in
  $\mathcal{P}_2(\R^2\times \R^2)$ under the Wasserstein metric $W_2$.
\end{Lemma}
\begin{proof}
  The martingale property $\gamma(y|x) = x$ of $\gamma\in \Gamma(\nu)$
  is equivalent to the identity
  \begin{displaymath}
    \int f(x) (y-x) d\gamma = 0
  \end{displaymath}
  for every bounded and continuous function $f=f(x)$ on $\R^2$.  The
  convexity and the closedness of $\Gamma(\nu)$ under $W_2$ readily
  follow.  It only remains to be shown that $\Gamma(\nu)$ is
  pre-compact under $W_2$ or, equivalently, that
  \begin{displaymath}
    \sup_{\gamma \in \Gamma(\nu)} \int (\abs{x}^2 + \abs{y}^2)
    \ind{\abs{x} +\abs{y} \geq K} 
    d\gamma \to 0, \quad K\to \infty. 
  \end{displaymath}
  For $\gamma \in \Gamma(\nu)$ we have that
  \begin{align*}
    \frac12 \int (\abs{x}^2 + \abs{y}^2) \ind{\abs{x} + \abs{y} \geq
    2K} d\gamma &\leq \int \left(\abs{x}^2 \ind{\abs{x}\geq K} +
                  \abs{y}^2
                  \ind{\abs{y} \geq K}\right) d\gamma \\
                &= \int \abs{\gamma(y|x)}^2 \ind{\abs{x}\geq K} d\gamma + \int
                  \abs{y}^2
                  \ind{\abs{y} \geq K} d\nu \\
                &\leq \int \abs{y}^2 \ind{\abs{x}\geq K} d\gamma + \int \abs{y}^2
                  \ind{\abs{y} \geq K} d\nu
  \end{align*}
  and then that
  \begin{align*}
    \int \abs{y}^2 \ind{\abs{x}\geq K} d\gamma &\leq \int \abs{y}^2
                                                 \ind{\abs{x} \geq K > \abs{y}^2} d\gamma + \int \abs{y}^2
                                                 \ind{\abs{y}^2 \geq K} d\nu  \\
                                               &\leq  K \gamma(\abs{x}\geq K) + \int \abs{y}^2 \ind{\abs{y} \geq
                                                 \sqrt{K}} d\nu.
  \end{align*}
  Finally, we obtain that
  \begin{align*}
    K \gamma(\abs{x}\geq K) \leq \frac1K\int \abs{x}^2 d\gamma =
    \frac1K\int \abs{\gamma(y|x)}^2 d\gamma \leq \frac1K \int
    \abs{y}^2 d\nu,
  \end{align*}
  and the result follows.
\end{proof}

\begin{Lemma}
  \label{lem:2}
  An optimal plan $\gamma$ for~\eqref{eq:5} exists.
\end{Lemma}

\begin{proof}
  Let $(\gamma_n)$ be a sequence in $\Gamma(\nu)$ such that
  \begin{displaymath}
    \lim_{n\to \infty} \int c(x,y) d\gamma_n = \inf_{\zeta \in
      \Gamma(\nu)} \int c(x,y) d\zeta. 
  \end{displaymath}
  By Lemma~\ref{lem:1}, $\Gamma(\nu)$ is compact under $W_2$. Hence,
  there is a subsequence $(\gamma_{n_k}) \subset (\gamma_n)$ that
  converges to $\gamma\in \Gamma(\nu)$ under $W_2$. Since the cost
  function $c=c(x,y)$ is continuous and has quadratic growth, we
  deduce that
  \begin{displaymath}
    \int c(x,y) d\gamma = \lim_{k\to \infty} \int c(x,y) d\gamma_{n_k}
    = \inf_{\zeta \in \Gamma} \int c(x,y) d\zeta. 
  \end{displaymath}
  Thus, $\gamma$ is an optimal plan.
\end{proof}

The implication~\ref{item:1}$\implies$\ref{item:2} of
Theorem~\ref{th:1} is proved in Lemma~\ref{lem:4} and relies on the
following first-order optimality condition.

\begin{Lemma}
  \label{lem:3}
  Let $\gamma$ be an optimal plan for~\eqref{eq:5}. Then
  \begin{equation}
    \label{eq:12}
    \int c(x,y) d\eta 
    \leq \int y_1 y_2 d\eta - \int y_1 d\eta \int y_2 d\eta, 
  \end{equation} 
  for every $\eta\in \mathcal{P}_2(\R^2\times \R^2)$ such that
  $\supp \eta \subset \supp \gamma$.
\end{Lemma}
\begin{proof}
  We first establish~\eqref{eq:12} for a Borel probability measure
  $\eta$ on $\R^2\times \R^2$ that has a bounded density with respect
  to $\gamma$:
  \begin{displaymath}
    V(x,y) = \frac{d\eta}{d\gamma} \in
    \mathcal{L}^\infty(\R^2\times \R^2).  
  \end{displaymath}
  We choose a non-atom $q \in \R^2$ of $\mu(dx) = \gamma(dx,\R^2)$ and
  define the probability measure
  $$
  \zeta(dx, dy)= \delta_{q}(dx) \eta(\R^2,dy),
  $$
  where $\delta_q$ is the Dirac measure concentrated at $q$.  For
  sufficiently small $\varepsilon>0$ the probability measure
  $$
  \widetilde \gamma = \gamma + \varepsilon (\zeta - \eta)
  $$
  is well-defined and has the same $y$-marginal $\nu$ as $\gamma$.  We
  define the conditional expectation
  $\widetilde X(x) = \widetilde\gamma(y|x)$ and observe that the law
  of $(\widetilde X, y)$ under $\widetilde \gamma$ belongs to
  $\Gamma(\nu)$. The optimality of $\gamma$ for~\eqref{eq:5} or,
  equivalently, for~\eqref{eq:6} implies that
  \begin{equation}
    \label{eq:13} 
    \int \widetilde X_1 \widetilde X_2 d\widetilde \gamma 
    \leq \int x_1x_2 d\gamma.  
  \end{equation}

  Standard computations based on Bayes formula show that
  \begin{align*}
    \widetilde X(x) = \ind{x\not=q}\frac{x -
    \varepsilon R(x)}{1 - \varepsilon U(x)}  + \ind{x=q} m,
  \end{align*}
  where $U(x) = \gamma(V|x)$, $R(x) = \gamma(Vy|x)$, and
  $m = \int y d\eta$.  Since $\abs{V}\leq K$ for some constant $K>0$,
  we deduce that $\abs{U}\leq K$ and
  $\abs{R} \leq K \gamma(\abs{y}|x)$. It follows that
  \begin{align*}
    \int \widetilde X_1 \widetilde X_2 d\widetilde \gamma & = \int
                                                            \frac{(x_1 - \varepsilon R_1)(x_2 - \varepsilon
                                                            R_2)}{(1-\varepsilon U)^2} (1-\varepsilon V) d\gamma
                                                            + \varepsilon m_1m_2 \\
                                                          &= \int \frac{(x_1 - \varepsilon R_1)(x_2 - \varepsilon
                                                            R_2)}{1-\varepsilon U}
                                                            d\gamma + \varepsilon m_1m_2 \\
                                                          &= \int x_1 x_2 d\gamma + \varepsilon \left(m_1m_2 + \int (x_1x_2
                                                            U - x_1 R_2 - x_2 R_1) d\gamma\right) + O(\varepsilon^2).
  \end{align*}
  In view of~\eqref{eq:13}, the first-order term is negative.  It can
  be written as
  \begin{align*}
    0 & \geq  m_1m_2 + \int (x_1x_2 U - x_1 R_2 - x_2 R_1) d\gamma \\
      & = \int y_1 d\eta\int y_2 d\eta + \int (x_1x_2 - x_1y_2 -x_2 y_1) V d\gamma  \\
      & = \int y_1 d\eta\int y_2 d\eta + \int (x_1x_2 - x_1y_2 -x_2 y_1) d\eta  \\
      & = \int y_1 d\eta\int y_2 d\eta + \int (c(x,y) - y_1y_2) d\eta
  \end{align*}
  and the result follows.

  In the general case, where $\eta\in \mathcal{P}_2(\R^2\times \R^2)$
  and $\supp{\eta} \subset \supp{\gamma}$, we use the approximation
  result from Appendix~\ref{sec:density-cw_2rd}. By
  Theorem~\ref{th:8}, there are Borel probability measures $(\eta_n)$
  on $\R^2\times \R^2$ that have bounded densities with respect to
  $\gamma$ and converge to $\eta$ under $W_2$. By what we have already
  proved,
  \begin{displaymath}
    \int c(x,y) d\eta_n  
    \leq \int y_1 y_2 d\eta_n - \int y_1 d\eta_n \int y_2 d\eta_n,
    \quad  n\geq 1. 
  \end{displaymath}
  Since the integrands are continuous functions with quadratic growth,
  we can pass to the limit as $n\to \infty$ and obtain~\eqref{eq:12}.
\end{proof}

\begin{Lemma}
  \label{lem:4}
  Let $\gamma$ be an optimal plan for~\eqref{eq:5}. Then
  condition~\ref{item:2} of Theorem~\ref{th:1} holds.
\end{Lemma}
 
\begin{proof}
  Lemma~\ref{lem:3} yields inequality~\eqref{eq:12} for the
  probability measure
  $$
  \eta(dx,dy) = (1-t)\delta_{(x^0,y^0)}(dx,dy) + t
  \delta_{(x^1,y^1)}(dx,dy),
  $$
  where $t\in [0,1]$, $(x^i,y^i) \in \supp{\gamma}$, and
  $\delta_{(x^i,y^i)}$ is the Dirac measure concentrated at
  $(x^i,y^i)$, $i=0,1$. Elementary computations show that for such
  $\eta$ \eqref{eq:12} becomes~\eqref{eq:7}.
\end{proof}

The equivalence of assertions~\ref{item:2} and~\ref{item:3} of
Theorem~\ref{th:1} is a special case of Lemma~\ref{lem:6}, whose proof
relies on the following geometric interpretation
of~\eqref{eq:7}. Figure~\ref{fig:1} visualizes the arguments.

\begin{Lemma}
  \label{lem:5}
  Let $x^i$ and $y^i$, $i=0,1$, be points in $\R^2$ such that
  $y^0_1 < y^1_1$. Then~\eqref{eq:7} holds if and only if for all
  $a_i < c(x^i,y^i)$, $i=0,1$, the graphs of the hyperbolas
  \begin{align*}
    h^0(s) &= y^0_2 + \frac{a_0}{s-y^0_1}, \quad s>y^0_1, \\
    h^1(s) &= y^1_2 + \frac{a_1}{s-y^1_1}, \quad s<y^1_1,
  \end{align*}
  do not intersect:
  $$
  h^0(s) < h^1(s), \quad s\in (y^0_1, y^1_1).
  $$
\end{Lemma}
\begin{proof}
  The result follows from the identity:
  \begin{align*}
    &(h^1(s) - h^0(s))(s-y^0_1)(y^1_1 - s)  \\
    &\qquad = (y^1_2 - y^0_2)(s-y^0_1)(y^1_1-s) - a_0(y^1_1-s) -
      a_1(s-y^0_1) \\
    &\qquad = (y^1_1 - y^0_1) \left(t(1-t)c(y^0,y^1) - (1-t) a_0 - t
      a_1 \right),
  \end{align*}
  where $s\in (y^0_1,y^1_1)$ and $t = (s - y^0_1)/(y^1_1 - y^0_1)$.
\end{proof}

\begin{Lemma}
  \label{lem:6}
  For a set $A\subset \R^2\times \R^2$ the following conditions are
  equivalent:
  \begin{enumerate}[label = {\rm (\roman{*})}, ref={\rm (\roman{*})}]
  \item \label{item:4} If points $(x^0,y^0)$ and $(x^1, y^1)$ belong
    to $A$, then~\eqref{eq:7} holds.
  \item \label{item:5} There is $G\in \mathfrak{M}$ such that
    \begin{displaymath}
      c(x,y) \leq \phi_G(y), \quad (x,y)\in A.
    \end{displaymath}
  \end{enumerate}
\end{Lemma}
 
\begin{proof}
  We observe first that under either~\ref{item:4} or~\ref{item:5},
  \begin{displaymath}
    c(x,y) \leq 0, \quad (x,y) \in A. 
  \end{displaymath}
  Indeed, under~\ref{item:4} this inequality follows
  from~\eqref{eq:7}, while under~\ref{item:5} it holds because
  $\phi_G\leq 0$. We define the open sets
  \begin{align*}
    B^i = \cup_{(x,y) \in A} B^i(x,y),\quad i=0,1,
  \end{align*}
  where, for $(x,y) \in A$,
  \begin{align*}
    B^0(x,y) & = \descr{z\in \R^2}{c(z,y) < c(x,y), \; z_1 >
               y_1} \\
             & = \descr{z\in \R^2}{z_2 < y_2 + \frac{c(x,y)}{z_1-y_1}, \;
               z_1 > y_1}, \\
    B^1(x,y) & = \descr{z\in \R^2}{c(z,y) < c(x,y), \; z_1 <
               y_1} \\
             & = \descr{z\in \R^2}{z_2> y_2 + \frac{c(x,y)}{z_1-y_1}, \; z_1 <
               y_1}.
  \end{align*}
  The boundaries of $B^i$, $i=0,1$, are, respectively, upper and lower
  envelopes of the graphs of increasing hyperbolas and thus, are
  maximal monotone sets.

  By Lemma~\ref{lem:5}, item~\ref{item:4} holds if and only if the
  sets $B^0$ and $B^1$ are disjoint:
  \begin{equation}
    \label{eq:14}
    B^0 \cap B^1 = \emptyset.   
  \end{equation}
  On the other hand, item~\ref{item:5} holds if and only if the closed
  set
  \begin{displaymath}
    C = \cap_{(x,y)\in A} \descr{z\in \R^2}{c(x,y) \leq c(z,y)} 
  \end{displaymath}
  contains a maximal monotone set $G$. As
  \begin{displaymath}
    C = \R^2\setminus \left(B^0 \cup B^1 \right),
  \end{displaymath}
  every such set $G$ separates $B^0$ and $B^1$. Hence, its existence
  yields~\eqref{eq:14} and then~\ref{item:4}. Conversely, if the sets
  $B^0$ and $B^1$ are disjoint, then their boundaries belong to
  $C$. As the boundaries are maximal monotone sets, we
  obtain~\ref{item:5}.
\end{proof}

The remaining assertions of the theorem follow from
Lemma~\ref{lem:8}. A key role is played by inequality~\eqref{eq:15}.

\begin{Lemma}
  \label{lem:7}
  Let $\gamma \in \Gamma(\nu)$ and $\mu$ be its $x$-marginal. For
  every $G\in \mathfrak{M}$ we have that
  \begin{displaymath}
    \gamma(\phi_G(y) - c(x,y)|x) \leq \phi_G(x) \leq 0, \quad \gamma\text{-a.s.}, 
  \end{displaymath}
  and then that
  \begin{equation}
    \label{eq:15}
    \int \phi_G(y) d\nu - \int c(x,y) d\gamma \leq \int \phi_G(x) d\mu
    \leq 0.
  \end{equation}
\end{Lemma}

\begin{proof}
  We only need to prove the inequality with conditional
  expectations. From Lemma~\ref{lem:16} we know that $\phi_G\leq 0$.
  As $x=\gamma(y|x)$, we deduce that for every $r\in \R^2$:
  \begin{displaymath}
    \gamma(c(y,r) - c(x,y)|x) = c(x,r), \quad \gamma\text{-a.s.},  
  \end{displaymath}
  and taking $\inf$ over a dense countable set of $r\in G$ obtain the
  result.
\end{proof}

The following lemma completes the proof of the theorem.

\begin{Lemma}
  \label{lem:8}
  Let $\gamma \in \Gamma(\nu)$, $\mu$ be its $x$-marginal, and
  $G\in \mathfrak{M}$ be such that
  \begin{equation}
    \label{eq:16}
    \int c(x,y) d\gamma \leq \int \phi_G(y) d\nu.
  \end{equation}
  Then in~\eqref{eq:16} we actually have the equality, $\gamma$ is an
  optimal plan for~\eqref{eq:5}, the set $G$ contains $\supp{\mu}$,
  and
  \begin{equation}
    \label{eq:17}
    c(x,y) = \phi_G(y) = \min_{r\in G} c(r,y), \quad (x,y)\in \supp{\gamma}.
  \end{equation} 
\end{Lemma}
\begin{proof}
  We shall write $\phi$ for $\phi_G$.  From~\eqref{eq:15}
  and~\eqref{eq:16} we deduce that $\gamma$ is a solution
  to~\eqref{eq:5}, that in~\eqref{eq:16} we have an equality, and that
  \begin{displaymath}
    \int \phi(x) d\mu = 0.
  \end{displaymath}
  Lemma~\ref{lem:16} states that $\phi\leq 0$ and
  $G=\descr{x\in \R^2}{\phi(x) = 0}$. It follows that $\mu(G)=1$.
  Being a closed set, $G$ contains $\supp{\mu}$. In particular, if
  $(x,y)\in \supp{\gamma}$, then $x\in G$.  It follows that
  \begin{displaymath}
    c(x,y) \geq \inf_{r\in G} c(r,y) = \phi(y), \quad (x,y) \in
    \supp{\gamma}. 
  \end{displaymath}
  Accounting for~\eqref{eq:16}, we deduce that
  \begin{displaymath}
    c(x,y) = \phi(y), \quad \gamma\text{-a.s.}. 
  \end{displaymath}
  Hence, for every $(x,y)\in \supp{\gamma}$ we can find a sequence
  $\braces{(x_n,y_n)} \subset \supp{\gamma}$ that converges to $(x,y)$
  and such that $c(x_n,y_n) = \phi(y_n)$, $n\geq 1$.  Being a
  pointwise infinum of continuous functions, the function $\phi$ is
  upper semi-continuous. It follows that
  \begin{displaymath}
    c(x,y) = \lim_{n\to \infty} c(x_n,y_n) = \lim_{n\to \infty}
    \phi(y_n) \leq \phi(y),   
  \end{displaymath}
  and we obtain~\eqref{eq:17}.
\end{proof}

\section{Dual problem}
\label{sec:dual-probl}

In view of Theorem~\ref{th:1} and Lemma~\ref{lem:7}, a natural dual
problem to~\eqref{eq:5} is to
\begin{equation}
  \label{eq:18}
  \text{maximize} \quad \int \phi_G d\nu \quad \text{over}\quad G\in
  \mathfrak{M}.  
\end{equation}
Such problem appears in \citet{RochVila:94} in connection to their
study of Kyle-type equilibrium with insider;
see~Section~\ref{sec:equil-with-insid}. They use a direct method based
on the properties of the space of closed graph correspondences and
assume that $\nu$ has a compactly supported density.

We recall that $G = \descr{x\in \R^2}{\phi_G(x) = 0}$,
$G\in \mathfrak{M}$, and thus the family $\mathfrak{M}$ of all maximal
monotone sets in $\R^2$ is in one-to-one correspondence with the
family of functions
\begin{displaymath}
  \Phi \set \descr{\phi_G}{G\in \mathfrak{M}}.
\end{displaymath}
Hence, ~\eqref{eq:18} is equivalent to the problem, where we
\begin{displaymath}
  \text{maximize} \quad \int \phi d\nu \quad \text{over}\quad \phi\in
  \Phi.
\end{displaymath}
A technical inconvenience of the set $\Phi$ is the absence of
convexity. It turns out that the set of functions dominated by the
elements of $\Phi$ is not only convex, but also admits a
self-contained description related to item~\ref{item:2} of
Theorem~\ref{th:1}.
\begin{Lemma}
  \label{lem:9}
  Let $\map{\phi}{\R^2}{[-\infty,0]}$ be a Borel function. Then
  $\phi\leq \phi_G$ for some $G\in \mathfrak{M}$ if and only if
  \begin{equation}
    \label{eq:19}
    (1-t) \phi(y^0) + t \phi(y^1) \leq t (1-t) c(y^0, y^1), \quad  
    y^0,y^1\in \R^2, \;  t\in [0,1].    
  \end{equation}
  The set of such functions $\phi$ is convex.
\end{Lemma}
\begin{proof}
  The result follows directly from Lemma~\ref{lem:6}, where we take
  \begin{displaymath}
    A = \descr{(x,y)\in \R^2\times \R^2}{c(x,y) = \phi(y)}. 
  \end{displaymath} 
  Clearly, the family of functions $\phi$ satisfying~\eqref{eq:19} is
  convex.
\end{proof}

\begin{Theorem}
  \label{th:2}
  Let $\nu \in \mathcal{P}_2(\R^2)$.  We have that
  \begin{displaymath}
    \min_{\gamma \in \Gamma(\nu)} \int c(x,y) d\gamma = \max_{G\in
      \mathfrak{M}} \int \phi_G d\nu, 
  \end{displaymath}
  where the lower and upper bounds are attained at respective
  solutions to~\eqref{eq:5} and~\eqref{eq:18}. A probability measure
  $\gamma \in \Gamma(\nu)$ and a maximal monotone set $G$ are such
  solutions if and only if
  \begin{equation}
    \label{eq:20}
    c(x,y) = \phi_G(y), \quad (x,y)\in \supp{\gamma}. 
  \end{equation}
  In this case, $G$ contains the support of the $x$-marginal of
  $\gamma$. Moreover, $\phi_G$ and $G$ are uniquely defined on
  $\supp{\nu}$, that is,
  \begin{align*}
    \phi_{G}(y) &= \phi_{\widetilde{G}}(y), \quad y\in \supp{\nu}, \\
    G\cap \supp{\nu} &= \widetilde G \cap \supp{\nu},
  \end{align*}
  for any other solution $\widetilde{G}$ to~\eqref{eq:18}. In
  particular, $\phi_G$ and $G$ are unique if $\supp{\nu}=\R^2$.
\end{Theorem}

\begin{proof}
  With an exception of the uniqueness part, all other assertions
  follow directly from Theorem~\ref{th:1} and Lemmas~\ref{lem:7}
  and~\ref{lem:8}.

  Let $\gamma$ be a solution to~\eqref{eq:5}, $G$ and $\widetilde G$
  be solutions to~\eqref{eq:18} and denote $\phi= \phi_{G}$ and
  $\widetilde \phi = \phi_{\widetilde G}$.  From~\eqref{eq:20} we
  deduce that the functions $\phi$ and $\widetilde{\phi}$ coincide on
  $P_y$, the projection of $\supp{\gamma}$ on $y$-coordinates. Since
  every $y\in \supp{\nu}$ is the limit of a sequence
  $(y_n) \subset P_y$, Lemma~\ref{lem:18} yields that
  \begin{displaymath}
    \phi(y) = \lim_{n\to \infty} \phi(y_n) = \lim_{n\to \infty}
    \widetilde{\phi}(y_n) = \widetilde{\phi}(y).
  \end{displaymath}
  We have proved the uniqueness of $\phi_G$ on $\supp{\nu}$. The
  uniqueness of $G$ on $\supp{\nu}$ holds as
  $G = \descr{x\in \R^2}{\phi_G(x) = 0}$.
\end{proof}

\section{Optimal maps}
\label{sec:optimal-map}

For simplicity of notations, we slightly modify the setup.  We start
with a 2-dimensional random variable $Y=(Y_1,Y_2)$ having a finite
second moment:
$Y\in \mathcal{L}^2 = \mathcal{L}^2(\Omega, \mathcal{F}, \mathbb{P})$.
As usual, we identify random variables that differ only on a set of
measure zero.  Our goal is to
\begin{equation}
  \label{eq:21}
  \text{minimize} \quad \EP{c(X,Y)} 
  \quad \text{over}\quad X \in \mathcal{X}(Y)   
\end{equation} 
for the same cost function $c(x,y) = (x_1-y_1)(x_2-y_2)$ and the
domain
\begin{displaymath}
  \mathcal{X}(Y) \set \descr{X=(X_1,X_2)\in \mathcal{L}^2}{X \text{ is
      $Y$-measurable and } \cEP{Y}{X} = X}.  
\end{displaymath}
We denote $\nu = \law{Y}$ and observe that $\law{X,Y} \in \Gamma(\nu)$
for every $X\in \mathcal{X}(Y)$.  Thus, optimal \emph{plan}
problem~\eqref{eq:5} may be viewed as a Kantorovich-type relaxation of
optimal \emph{map} problem~\eqref{eq:21}. In general,
\begin{equation}
  \label{eq:22}
  \min_{\gamma\in \Gamma(\nu)} \int c(x,y) d\gamma \leq
  \inf_{X \in \mathcal{X}(Y)} \EP{c(X,Y)},
\end{equation}
and the inequality may be strict and an optimal map may not exist as
Examples~\ref{ex:2} and~\ref{ex:3} show.

The main results of this sections are Theorems~\ref{th:3}, \ref{th:4},
and \ref{th:5}. Theorem~\ref{th:3} yields the equality
in~\eqref{eq:22} provided that $\nu=\Law(Y)$ is atomless.
Theorem~\ref{th:4} shows the existence of an optimal map if $\nu$ is
$\mathcal{D}$-regular in the sense of
Definition~\ref{def:1}. Theorem~\ref{th:5} establishes the uniqueness
of optimal plan and map if, in addition, every component $Y_i$ has a
continuous distribution function.  The last two theorems play a key
role in the study of equilibrium in
Section~\ref{sec:equil-with-insid}.

We shall use the notations from
Appendix~\ref{sec:dual-space-functions} related to the function
$\phi = \phi_G$, where $G\in\mathfrak{M}$.  In particular,
$D^c = (D^c_1, D^c_2)$ stands for the differential operator associated
with the cost function $c=c(x,y)$:
\begin{align*}
  D^c_1\phi(y) = y_1 - \frac{\partial \phi}{\partial y_2}(y), \quad
  D^c_2\phi(y) = y_2 - \frac{\partial \phi}{\partial y_1}(y), \quad
  y \in \dom{\nabla \phi}, 
\end{align*}
where $\dom{\nabla \phi}$ is the set of points where $\phi$ is
differentiable.  We denote by $E^G = E_1^G \cup E_2^G$ the union of
the vertical and horizontal line segments of $G$:
\begin{equation}
  \label{eq:23}
  \begin{split}
    E_i^G(t) &= \descr{x=(x_1,x_2)\in G}{x_i=t}, \quad t\in \R,\\
    \mathcal{T}_i^G &= \descr{t\in \R}{{E_i^G(t)} \;\text{has more
        than one point}}, \\
    E_i^G & = \cup_{t\in \mathcal{T}_i^G} E_i^G(t), \quad i=1,2.
  \end{split}
\end{equation}
Clearly, the sets $(\mathcal{T}^G_i)$ are countable. Finally, we
define
\begin{align*}
  \Arg_G(y) & = \argmin_{x\in G}c(x,y) = \descr{x\in G}{\phi_G(y) =
              c(x,y)}, \\ 
  \dom{\Arg_G} & = \descr{y\in \R^2}{\Arg_G(y) \not=\emptyset}.
\end{align*}

The following result is similar to that of~\citet{Prat:07} obtained
for the classical unconstrained optimal transport problem.

\begin{Theorem}
  \label{th:3}
  Let $Y=(Y_1,Y_2)\in \mathcal{L}^2$ and suppose that $\nu = \law{Y}$
  is atomless. Then plan and map problems~\eqref{eq:5} and
  \eqref{eq:21} have identical values:
  \begin{displaymath}
    \min_{\gamma\in \Gamma(\nu)} \int c(x,y) d\gamma =
    \inf_{X \in \mathcal{X}(Y)} \EP{c(X,Y)}.
  \end{displaymath}
\end{Theorem}

The proof of the theorem relies on some lemmas.

\begin{Lemma}
  \label{lem:10}
  Let $\nu\in \mathcal{P}_2(\R^2)$, $\gamma\in \Gamma(\nu)$ be an
  optimal plan for~\eqref{eq:5}, and $G\in \mathfrak{M}$ be a
  maximizer for~\eqref{eq:18}. If $\nu(G) > 0$, then the probability
  measure
  \begin{displaymath}
    \eta(dx,dy) = \frac1{\nu(G)} \ind{y\in G} \gamma(dx,dy) 
  \end{displaymath}
  has the martingale property: $\eta(y|x) = x$.
\end{Lemma}

\begin{proof}
  We write $\phi$ for $\phi_G$. We shall show that
  $\eta(y_1|x) = x_1$, that is, that
  \begin{equation}
    \label{eq:24}
    \int f(x) (y_1-x_1) d\eta = \frac1{\nu(G)} \int
    f(x) (y_1-x_1) \ind{y\in G} d\gamma = 0,    
  \end{equation}
  for every bounded Borel function $f$ on $\R^2$. The martingale
  property for the second coordinate has a similar proof.

  Let $E_2 = E_2 ^G = \cup_{t\in \mathcal{T}_2}E_2(t)$ be the union of
  the horizontal line segments of $G$.  If $(x,y)\in \supp{\gamma}$,
  then Theorem~\ref{th:2} yields that $x\in G$ and $c(x,y) =
  \phi(y)$. If, in addition, $y\in G\setminus E_2$, then
  $c(x,y)= \phi(y) = 0$ and, subsequently,
  $x_1=y_1$. Hence,~\eqref{eq:24} holds if
  \begin{equation}
    \label{eq:25}
    \int f(x) (y_1-x_1) \ind{y\in E_2(t)} d\gamma  = 0, \quad t\in
    \mathcal{T}_2. 
  \end{equation}

  Hereafter, we fix $t\in \mathcal{T}_2$. Let
  $(x,y)\in \supp{\gamma}$. If $y\in E_2(t)$, then
  $c(x,y) = \phi(y) = 0$ and thus, $x\in E_2(t)$. Conversely, if
  $x\in \ri{E_2(t)}$, the relative interior of $E_2(t)$, then
  Lemma~\ref{lem:19} yields that $y\in G$ and then, as
  $c(x,y) = \phi(y) = 0$, that $y\in E_2(t)$.  Hence,
  \begin{align*}
    \ind{y\in E_2(t)}
    &= \ind{x,y\in E_2(t)} \\
    &= \ind{x\in \ri{E_2(t)}} + \ind{x=a(t),y\in E_2(t)} +
      \ind{x=b(t),y\in E_2(t)},
  \end{align*}
  where $a(t)$ and $b(t)$ are the boundary points of $E_2(t)$ such
  that $a_1(t)<b_1(t)$.  Accounting for the martingale property of
  $\gamma$, we obtain that
  \begin{align*}
    \int f(x) (y_1-x_1) \ind{x\in \ri{E_2(t)}} d\gamma = 0.
  \end{align*}
  
  Let $y\in \R^2$ be such that $\phi(y) = c(b(t),y)$.  If
  $y\not\in G$, then Lemma~\ref{lem:19} yields that $b_1(t) > y_1$. If
  $y\in G\setminus E_2(t)$, then $c(b(t),y)= \phi(y)=0$ and thus,
  $b_1(t) = y_1$. Finally, if $y\in E_2(t)$, then $b_1(t) \geq
  y_1$. It follows that
  \begin{align*}
    \int \abs{x_1-y_1} \ind{x=b(t),y\in E_2(t)} d\gamma \leq \int
    (x_1-y_1) \ind{x=b(t)} d\gamma = 0,
  \end{align*}
  where at the last step we used the martingale property of
  $\gamma$. The case of the left boundary $a(t)$ is similar. We have
  proved~\eqref{eq:25}.
\end{proof}

\begin{Lemma}
  \label{lem:11}
  Let $G\in \mathfrak{M}$ and $X=(X_1,X_2)$ and $Y=(Y_1,Y_2)$ be
  random variables such that $X$ takes values in $G$,
  $X\ind{Y\in G} = Y\ind{Y\in G}$, and $c(X,Y) = \phi_G(Y)$. If the
  law of $Y$ is atomless, then for every $\epsilon>0$ there is a
  random variable $Z=Z(\epsilon)$ such that $\law{Z} = \law{Y}$,
  $\abs{Z - Y}\leq \epsilon$, and $X$ is $Z$-measurable.
\end{Lemma}
\begin{proof}
  We fix $\epsilon>0$ and denote $\phi = \phi_G$, $\Arg = \Arg_G$, and
  \begin{displaymath}
    D = \dom{\Arg}\setminus (\dom{\nabla \phi} \cup G). 
  \end{displaymath}
  Theorems~\ref{th:9} and~\ref{th:10} show that $D = \cup_n D_n$,
  where $D_n$ is either a point or the graph of a strictly decreasing
  function.  Of course, we can choose the sets $(D_n)$ so that
  \begin{displaymath}
    \diam{D_n} \set \sup_{x,y\in D_n} \abs{x-y} \leq \epsilon, \quad
    n\geq 1. 
  \end{displaymath}
  For every $n\geq 1$ we shall construct a two-dimensional random
  variable $Z^n=(Z^n_1,Z^n_2)$ and a Borel function
  $\map{f^n}{D_n}{G}$ such that
  \begin{equation}
    \label{eq:26}
    \begin{split}
      \law{Z^n \ind{Y\in D_n}} &= \law{Y \ind{Y\in D_n}}, \\
      X \ind{Y\in D_n} & = f^n(Z^n)\ind{Y\in D_n}.
    \end{split}
  \end{equation}
  Given the sequence of such pairs $(Z^n,f^n)$, $n\geq 1$, we define
  \begin{align*}
    Z &= Y \ind{Y\in \dom{\nabla \phi} \cup G}  + 
        \sum_n Z^n \ind{Y\in F_n}, \\
    f(y) &= y \ind{y \in G} + D^c\phi(y)  \ind{y \in
           \dom{\nabla\phi}\setminus G} + \sum_n 
           f^n(y)\ind{y\in F_n}.   
  \end{align*}
  where $F_n = D_n \setminus \cup_{k<n} D_k$.  We have that
  $\law{Z} = \law{Y} = \nu$ and $\abs{Z - Y} \leq \epsilon$. Moreover,
  in view of Theorem~\ref{th:9}, $X = f(Z)$.  Hence, \eqref{eq:26} is
  all we need to obtain.

  Using the conditional probabilities with respect to events
  $\braces{Y\in D_n}$, we can reduce the general case to the situation
  where
  \begin{displaymath}
    Y\in D = \descr{(t,h(t))}{t\in [t_0,t_1]},  
  \end{displaymath}
  for some strictly decreasing function $h=h(t)$. Since $\nu$ is
  atomless, every component $Y_i$ has a continuous distribution
  function $a_i(t) = \PP{Y_i\leq t}$, $i=1,2$. It follows that
  $U = a_1(Y_1)$ has the uniform distribution on $[0,1]$ and
  $Y_1 = a_1^{-1}(U)$, where $a_1^{-1}$ is the pseudo-inverse function
  to $a_1$:
  \begin{displaymath}
    a_1^{-1}(t) = \inf\descr{s\in [t_0,t_1]}{a_1(s) \geq t}, \quad t\in
    [0,1]. 
  \end{displaymath}
  In particular, $Y=(Y_1,h(Y_1))$ is $U$-measurable.
  
  Lemma~\ref{lem:29} yields Borel functions $\map{g_i}{D}{G}$,
  $i=1,2$, such that either $X=g_1(Y)$ or $X=g_2(Y)$.  As the
  functions
  \begin{align*}
    b_i(t) = \PP{U \leq t, X = g_i(Y)}, \quad t\in [0,1],  \; i=1,2, 
  \end{align*}
  are continuous and increasing, the random variable
  \begin{displaymath}
    V = b_1(U) \ind{X=g_1(Y)} + (b_1(1) + b_2(U)) \ind{X=g_2(Y)}
  \end{displaymath}
  has the uniform distribution on $[0,1]$. Clearly, $U$ and the
  indicators $(\ind{X=g_i(Y)})$ are $V$-measurable. It follows that
  $Y$ and $X$ are also $V$-measurable. Setting
  \begin{displaymath}
    Z_1 = a_1^{-1}(V), \quad Z_2 = h(Z_1),  
  \end{displaymath}
  we obtain that $Z=(Z_1,Z_2)$ has the same law as $Y$, that
  $V=a_1(Z_1)$, and that $X$ is $Z$-measurable.
\end{proof}

\begin{proof}[Proof of Theorem~\ref{th:3}]
  Let $\gamma$ be an optimal plan for~\eqref{eq:5}. By extending, if
  necessary, the underlying probability space we can assume that
  $\gamma = \law{X,Y}$ for some random variable $X$. As
  $\gamma(y|x) = x$, we have that $X = \cEP{Y}{X}$. Theorem~\ref{th:1}
  yields $G\in \mathfrak{M}$ such that $X\in G$ and
  $c(X,Y) = \phi_G(Y)$.

  We denote $\widetilde X = X \ind{Y\not\in G} + Y\ind{Y\in G}$ and
  observe that $\widetilde \gamma = \law{\widetilde X,Y}$ is another
  optimal plan. Indeed, by Lemma~\ref{lem:10},
  \begin{displaymath}
    \cEP{(Y-X) \ind{Y\in G}}{X} = 0
  \end{displaymath}
  and therefore, for a bounded Borel function $g=g(x)$ on $\R^2$,
  \begin{displaymath}
    \EP{(Y-\widetilde X)g(\widetilde X)} =  \EP{(Y-X)g(X) \ind{Y\not
        \in G}}= \EP{(Y-X) g(X)} = 0.
  \end{displaymath}
  It follows that $\cEP{Y}{\widetilde X} = \widetilde{X}$ and thus,
  $\widetilde \gamma \in \Gamma(\nu)$. By the construction of
  $\widetilde X$, we have that $c(X,Y) = c(\widetilde X,Y)$ and the
  optimality of $\widetilde \gamma$ follows.  This fact allows us to
  assume from the start that $X\ind{Y\in G} = Y \ind{Y\in G}$.  Then,
  $X$ and $Y$ satisfy the assumptions of Lemma~\ref{lem:11}.

  Let $\epsilon>0$ and $Z=Z(\epsilon)$ be the random variable yielded
  by Lemma~\ref{lem:11}. As $X$ is $Z$-measurable, the conditional
  expectation $V \set \cEP{Z}{X}$ is also $Z$-measurable. Thus, there
  is a Borel function $\map{f}{\R^2}{\R^2}$ such that $V = f(Z)$.
  Since $Y$ and $Z$ have identical laws, $U \set f(Y) =
  \cEP{Y}{U}$. As
  \begin{displaymath}
    \abs{V_i - X_i} = \abs{\cEP{Z_i - Y_i}{X}} \leq 
    \cEP{\abs{Z-Y}}{X} \leq 
    \epsilon, \quad i=1,2,   
  \end{displaymath}
  we deduce that
  \begin{align*}
    \EP{c(U,Y)} & = \EP{c(V, Z)}= \EP{Z_1Z_2} - \EP{V_1 V_2} \\
                &\leq \EP{Y_1 Y_2} - \EP{X_1 X_2} + \epsilon
                  \EP{\abs{X_1}+\abs{V_2}} \\
                &\leq \EP{c(X,Y)} + \epsilon \EP{\abs{Y_1}+\abs{Y_2}}.
  \end{align*}
  The result follows, because $\epsilon$ is any positive number.
\end{proof}

Let $\mathcal{D}$ be the family of graphs of strictly decreasing
functions $f=f(t)$ defined on closed intervals of $\R$ such that both
$f$ and its inverse $f^{-1}$ are Lipschitz functions:
\begin{displaymath}
  \frac1K (t-s) \leq f(s) - f(t) \leq K(t-s), \quad s<t,  
\end{displaymath}
for some constant $K=K(f)>0$.  To make statements shorter we allow for
a degenerate case where the domain of $f$ is just a point. Thus,
$\R^2\subset \mathcal{D}$.

\begin{Definition}
  \label{def:1}
  A Borel probability measure $\mu$ on $\R^2$ is
  \emph{$\mathcal{D}$-regular} if $\mu(D)=0$, $D\in \mathcal{D}$.
\end{Definition}

The following theorem establishes the existence of optimal maps that
induce optimal plans.

\begin{Theorem}
  \label{th:4}
  Let $Y=(Y_1,Y_2)\in \mathcal{L}^2$ and suppose that $\nu = \law{Y}$
  is $\mathcal{D}$-regular. Let $G\in \mathfrak{M}$ be a maximizer
  for~\eqref{eq:18} and denote $\phi = \phi_G$ and $E=E^G$. Then
  \begin{displaymath}
    X = Y\ind{Y\in E} + D^c\phi(Y)\ind{Y\not\in E}
  \end{displaymath}  
  is an optimal map for~\eqref{eq:21}, $\gamma = \law{X,Y}$ is an
  optimal plan for~\eqref{eq:5}, and the law of $X$ is
  $\mathcal{D}$-regular. Moreover, if $\widetilde X$ is an optimal map
  and $\widetilde \gamma$ is an optimal plan, then
  \begin{align}
    \label{eq:27}
    X\ind{Y\not\in E} & = \widetilde{X} \ind{Y\not\in E},\\
    \label{eq:28}
    \ind{y\not\in E} \gamma(dx,dy) & = \ind{y\not\in E}
                                     \widetilde\gamma(dx,dy).
  \end{align}
\end{Theorem}

\begin{proof}
  Let $\widetilde\gamma\in \Gamma(\nu)$ be an optimal plan.  From
  Theorem~\ref{th:2} we deduce that if
  $(x,y)\in \supp{\widetilde{\gamma}}$, then
  $y\in \dom{\Arg} = \dom{\Arg_G}$.  In particular,
  $\nu(\dom{\Arg}) = 1$. By Theorem~\ref{th:10}, the exception set
  $\dom{\Arg}\setminus \dom{\nabla \phi}$ belongs to the union of
  $E=E^G$ and of a countable family of sets from $\mathcal{D}$. Since
  $\nu = \law{Y}$ is $\mathcal{D}$-regular, we have that
  \begin{equation}
    \label{eq:29}
    \nu(\dom{\nabla \phi}\setminus E) = \nu(\dom{\Arg}\setminus E) = 1 - \nu(E). 
  \end{equation}
  It follows that the random variable $X = (X_1,X_2)$ is well-defined.

  Theorem~\ref{th:9} shows that if
  $y\in \dom{\nabla \phi} \setminus E$, then $D^c \phi(y)$ is the only
  element of $\Arg(y)$.  It follows that $X\in G$ and
  $\phi(Y) = c(X,Y)$.  In view of Theorem~\ref{th:2},
  $\gamma\set \law{X,Y}$ is an optimal plan if it has the martingale
  property: $\gamma(y|x)= x$.
 
  If $(x,y) \in \supp{\widetilde \gamma}$ and
  $y\in \dom{\nabla \phi} \setminus E$, then Theorems~\ref{th:2}
  and~\ref{th:9} yield that $x=D^c\phi(y)$. Since $\gamma$ and
  $\widetilde \gamma$ have common $y$-marginal $\nu$
  satisfying~\eqref{eq:29}, they coincide outside of $\R^2\times E$,
  that is, \eqref{eq:28} holds.

  Let $f=f(x)$ be a bounded Borel function on $\R^2$. As
  $(Y-X)\ind{Y\in G} = 0$, we deduce that
  \begin{align*}
    \int (y-x) f(x) d\gamma &= \int (y-x) f(x) \ind{y\not\in G}
                              d\gamma = \int (y-x) f(x) \ind{y\not\in G}
                              d\widetilde \gamma \\
                            &= -\int (y-x) f(x) \ind{y\in G} d\widetilde \gamma=0,
  \end{align*}
  where the last two equalities follow from the martingale property of
  $\widetilde \gamma$ and Lemma~\ref{lem:10}, respectively. Thus,
  $\gamma(y|x) = x$. We have proved that $\gamma$ is an optimal plan
  and, in particular, that $X$ is an optimal map. The uniqueness
  property~\eqref{eq:27} for optimal maps follows directly from the
  corresponding property~\eqref{eq:28} for optimal plans.
 
  It only remains to be shown that $\mu \set \law{X}$ is
  $\mathcal{D}$-regular. As $\supp{\mu}\subset G$ and the intersection
  of $G$ with any set from $\mathcal{D}$ is a point, $\mu$ is
  $\mathcal{D}$-regular if and only if it is atomless. Assume to the
  contrary, that $\mu(\braces{r})>0$ for some $r\in G$ and define the
  Borel probability measure
  \begin{displaymath}
    \eta(dy) = \frac1{\mu(\braces{r})} \gamma(\braces{r},dy).
  \end{displaymath}
  Being $\mathcal{D}$-regular, the measure $\nu$ is atomless. Hence,
  \begin{align*}
    \mu(\braces{r})\eta(G) &=\gamma(\braces{r}\times G) = \PP{X=r,
                             Y\in G} \\
                           &= \PP{Y=r} = \nu(\braces{r}) = 0.
  \end{align*}
  From the optimality of $\gamma$ we deduce that
  \begin{displaymath}
    \supp{\eta} \subset D(r) \set \descr{y\in \R^2}{\phi(y) =
      c(r,y)} 
  \end{displaymath}
  and then that $\eta(D(r)\setminus G) = 1>0$. The martingale property
  of $\gamma$ yields that $\int y d\eta = r$. The last two properties
  of $\eta$ and the fact that $\phi<0$ outside of $G$ imply the
  existence of $y^0,y^1\in D(r)\setminus G$ such that
  \begin{displaymath}
    y^0_1 < r_1 < y^1_1, \quad y^1_2 < r_2 < y^0_2. 
  \end{displaymath}
  By Lemma~\ref{lem:20}, $D(r)$ belongs to the graph of a strictly
  decreasing linear function and thus, belongs to $\mathcal{D}$. As
  $\nu$ is $\mathcal{D}$-regular, we arrive to a contradiction:
  $\mu(\braces{r}) = \gamma(\braces{r}\times D(r)) \leq \nu(D(r)) =
  0$.
\end{proof}

We now state the main uniqueness result of the paper, which can be
viewed as an adaptation of the classical Brenier theorem to our
setting.  We point out that our regularity assumption on
$\nu$ is weaker than the standard condition of the Brenier theorem,
which requires $\nu$ to assign zero mass to rotations of the graphs of
Lipschitz functions.

\begin{Theorem}
  \label{th:5}
  Let $Y=(Y_1,Y_2)\in \mathcal{L}^2$ and suppose that $\nu = \law{Y}$
  is $\mathcal{D}$-regular and the (one-dimensional) laws of $Y_1$ and
  $Y_2$ are atomless. Let $G\in \mathfrak{M}$ be a maximizer
  for~\eqref{eq:18} and denote $\phi = \phi_G$. Then $X=D^c\phi(Y)$
  or, in more detail,
  \begin{align*}
    X_1 &= D^c_1\phi(Y)
          = Y_1 - \frac{\partial \phi}{\partial y_2}(Y), \\
    X_2 &= D^c_2\phi(Y) =Y_2 - \frac{\partial \phi}{\partial y_1}(Y),
  \end{align*}
  is the unique optimal map for~\eqref{eq:21} and the law of $(X,Y)$
  is the unique optimal plan for~\eqref{eq:5}. Moreover, the law of
  $X$ is $\mathcal{D}$-regular and the laws of $X_1$ and $X_2$ are
  atomless.
\end{Theorem}

\begin{proof}
  We omit $G$ from the notations~\eqref{eq:23} related to its vertical
  and horizontal line segments.  As the law of $Y_i$ is atomless and
  the set $\mathcal{T}_i$ is countable, we deduce that
  \begin{displaymath}
    \nu(E) = \PP{Y\in E} \leq \sum_{i=1}^2  \PP{Y \in E_i} =  \sum_{i=1}^2
    \PP{Y_i \in \mathcal{T}_i} = 0. 
  \end{displaymath}
  Except the continuity of the distribution functions for $X_1$ and
  $X_2$, all other assertions follow directly from Theorem~\ref{th:4}.

  We shall prove that the law of $X_2$ is atomless.  If
  $t\not\in \mathcal{T}_2$, then the set ${E}_2(t)$ is a singleton:
  $E_2(t) = \braces{z}$.  By Theorem~\ref{th:4}, the law of $X$ is
  $\mathcal{D}$-regular and, in particular, atomless. It follows that
  $\PP{X_2 = t} = \PP{X=z} = 0$.

  Let $t\in \mathcal{T}_2$.  Lemma~\ref{lem:19} shows that if
  $x\in \ri{E_2(t)}$ and $\phi(y) = c(x,y)$, then $c(x,y)=0$ and
  subsequently, $y\in E_2(t)$.  As $X\in G$, $\phi(Y) = c(X,Y)$, and
  the law of $X$ is atomless, we obtain that
  \begin{align*}
    \PP{X_2 = t} &= \PP{X \in E_2(t)} = \PP{X\in \ri{E_2(t)}} \\
                 &\leq \PP{Y\in E_2(t)} = \PP{Y_2 = t} = 0,
  \end{align*}
  where the last step holds by the continuity of the law of $Y_2$.
\end{proof}

\section{Examples}
\label{sec:examples}

\begin{Example}[Linear optimal map]
  \label{ex:1}
  Let $Y=(Y_1, Y_2)$ be a random variable in $\mathcal{L}^2$ such that
  $\EP{Y_i}=0$, $\EP{Y_i^2} = \sigma_i^2>0$, and
  \begin{displaymath}
    \cEP{Y_i}{Z}= \frac{\EP{Y_iZ}}{\EP{Z^2}} Z \quad\text{for all} \quad 
    Z=a_1Y_1 + a_2Y_2, \; a_j\in \R.   
  \end{displaymath}
  The latter property holds if the distribution of $Y$ is Gaussian or,
  more generally, elliptically contoured.
  
  We denote $\lambda = \frac{\sigma_2}{\sigma_1} > 0$ and define
  $G=\descr{x\in \R^2}{x_2=\lambda x_1}$ and
  \begin{align*}
    X_1 = \frac{1}{2}Y_1 +  \frac{1}{2\lambda} Y_2, \quad X_2 =
    \lambda X_1 = 
    \frac{\lambda}{2}Y_1 +  \frac{1}{2} Y_2. 
  \end{align*}
  Elementary computations show that $\cEP{Y}{X} = X = (X_1,X_2)$ and
  \begin{displaymath}
    \phi_G(Y) = \inf_{x\in G}c(x,Y) = \inf_{x_1\in \R}(Y_1 - x_1)(Y_2 -
    \lambda x_1) = c(X,Y). 
  \end{displaymath}
  Being the graph of an increasing linear function,
  $G\in \mathfrak{M}$.  Setting $\nu = \law{Y}$, we deduce from
  Theorem~\ref{th:2} that $G$ and $\gamma = \law{X,Y}$ are respective
  solutions to~\eqref{eq:18} and~\eqref{eq:5}.  Moreover, as $X$ is
  the only element of $G$ such that $\phi_G(Y) = c(X,Y)$, the
  characteristic property~\eqref{eq:20} yields that $\gamma$ is the
  unique optimal plan. In particular, $X$ is the unique optimal map
  for~\eqref{eq:21}.
\end{Example}

\begin{Example}[Optimal map may not yield optimal plan]
  \label{ex:2}
  Let $Y$ be a random variable taking values in $y^0=(-1,1)$,
  $y^1=(0,-1)$, and $y^2=(1,0)$ with probability $\frac13$.  Direct
  computations show that the points
  \begin{displaymath}
    z^i = \frac{1}{3} y^0 + \frac{2}{3}y^i = (-1)^{i}
    \big(\frac13,\frac13\big), \quad i=1,2, 
  \end{displaymath}
  belong to the set
  $G = \descr{x\in \R^2}{c(x,y^0) = - \frac89,\; x_1 >y_1^0}$, that
  \begin{displaymath}
    \phi_G(y^i) = \min_{x\in G} c(y^i,x) = c(y^i,z^i), \quad i=1,2,  
  \end{displaymath}
  and that the probability measure
  \begin{displaymath}
    \gamma =  \sum_{i=1}^2 \left(\frac{1}{3} \delta_{(z^i,y^i)} 
      + \frac{1}{6}\delta_{(z^i, y^0)}\right)
  \end{displaymath}
  belongs to $\Gamma(\nu)$, where $\nu = \law{Y}$. Being the graph of
  an increasing hyperbola, $G\in \mathfrak{M}$. By Theorem~\ref{th:1},
  $\gamma$ is an optimal plan for~\eqref{eq:5}. The value of this
  problem is
  \begin{displaymath}
    \int c(x,y) d\gamma = \sum_{i=1}^2 \left(\frac{1}{3} c(z^i,y^i)  +
      \frac{1}{6} c(z^i,y^0)\right) = - \frac{4}{9}.
  \end{displaymath}
    
  On the other hand, let $X\in \mathcal{X}(Y)$, that is, $X$ is
  $Y$-measurable and $X=\cEP{Y}{X}$. We write $x^i = X(y^i)$,
  $i=1,2,3$. If all $(x^i)$ are distinct, then $X=Y$ and $c(X,Y)=0$.
  If they are the same point, then $X=\EP{Y} = 0$ and
  \begin{displaymath}
    \EP{c(X,Y)} = \frac13 \sum_{i=0}^2 c(0,y^i) = \frac13 \sum_{i=0}^2
    y^i_1 y^i_2 = - \frac13.  
  \end{displaymath}
  Finally, if precisely two of the elements of $(x^i)$ coincide:
  $x^{k} = x^{l} \not = x^{m}$, where $(k,l,m)$ is a permutation of
  $(0,1,2)$, then $x^{k}=x^{l}=\frac{1}{2}(y^{k}+y^{l})$,
  $x^{m} = y^{m}$, and
  \begin{align*}
    \EP{c(X,Y)} = \frac13 \left(c(y^{k},\frac12 (y^{k} + y^{l}))
    + c(y^{l},\frac12 (y^{k} + y^{l}))\right)
    = \frac16 c(y^{k},y^{l}). 
  \end{align*}
  As $c(y^0,y^1) = c(y^0,y^2) = -2$ and $c(y^1,y^2) = 1$, the value
  function of the optimal map problem~\eqref{eq:21} is given by
  $-\frac{1}{3}$, which is strictly less than $-\frac{4}{9}$, the
  value of the optimal plan problem~\eqref{eq:5}.
\end{Example}

\begin{Example}[Optimal map may not exist]
  \label{ex:3}
  Let $U$ and $V$ be independent symmetric random variables in
  $\mathcal{L}^2$ with $U$ having a continuous distribution function
  and $V$ taking values in $\braces{-1,1}$.  We define a 2-dimensional
  random variable
  \begin{displaymath}
    Y = \big(\frac{U}{3}(1-2V),U(1+2V)\big)\ind{U<0} + \big(U(1+2V),
    \frac{U}{3}(1-2V)\big)\ind{U\geq 0}. 
  \end{displaymath}
  The components $Y_1$ and $Y_2$ have continuous distribution
  functions and, in particular, $\nu=\law{Y}$ is atomless. By
  Theorem~\ref{th:3}, the plan and map problems~\eqref{eq:5}
  and~\eqref{eq:21} have identical values. We shall prove that there
  is a unique optimal plan, which is not induced by a ($Y$-measurable
  martingale) map, and hence, shall show that an optimal map does not
  exist.

  To this end, we define a 2-dimensional random variable
  \begin{align*}
    X =  (\frac{1}{3}U,U)\ind{U < 0} + (U,\frac{1}{3}U) \ind{U\geq 0}. 
  \end{align*}
  We observe that $X$ takes values in the set
  \begin{displaymath}
    G = \braces{x_2 = 3x_1, \; x_1<0} \cup \braces{x_2 =\frac13
      x_1, \; x_1 \geq 0}  
  \end{displaymath}
  consisting of two upward-slopping lines and thus, belonging to
  $\mathfrak{M}$. Direct computations show that $\cEP{Y}{X} = X$ and
  \begin{displaymath}
    c(X,Y) = \phi_G(Y) \set \inf_{x\in G}c(x,Y). 
  \end{displaymath}  
  By Theorem~\ref{th:1}, the law of $(X,Y)$ is an optimal plan and $G$
  is a dual maximizer.  We shall proceed to show that this is the only
  optimal plan and that it is not induced by a map from
  $\mathcal{X}(Y)$.
   
  From the construction of $Y$ we deduce the equality of the sets:
  \begin{displaymath}
    \braces{V = -1} = \braces{Y_1 = -Y_2} = \braces{Y =
      (-\abs{U},\abs{U})}.  
  \end{displaymath}
  It follows that
  \begin{align*}
    \cEP{X}{Y}\ind{Y_1=-Y_2} &= \cEP{X}{\abs{U}}\ind{Y_1=-Y_2} =
                               \frac13 (\abs{U},-\abs{U})
                               \ind{Y_1=-Y_2} \\
                             & = -\frac13 Y \ind{Y_1=-Y_2} \not= X\ind{Y_1=-Y_2}.
  \end{align*}
  Hence, $X$ is not $Y$-measurable.

  Let $\gamma \in \Gamma(\nu)$ be an optimal plan and $\mu$ be its
  $x$-marginal.  By Theorem~\ref{th:2}, $\supp{\mu} \subset G$ and
  \begin{displaymath}
    c(x,y) = \phi(y), \quad (x,y) \in \supp{\gamma}. 
  \end{displaymath}
  The random variable $Y$ takes values in $F= F_1\cup F_2$, where
  \begin{align*}
    F_1 &= \braces{y_2 = -y_1, \; y_1< 0}, \\
    F_2 &= \braces{y_2 = -9 y_1 \text{ or } y_2 = -\frac19 y_1, \; y_1 \geq 0}. 
  \end{align*}
  Elementary computations show that for $x\in G$ the set of $y\in F$
  such that $c(x,y) = \phi_G(y)$ consists of two points $g(x)$ and
  $f(x)$ such that
  \begin{align*}
    f(x) &= (-x_1,9x_1)\ind{x_1<0} + (3x_1,-\frac13x_1) \ind{x_1\geq
           0}, \\ 
    g(x) &= (3x_1,-3x_1) \ind{x_1 < 0} + (-x_1,x_1)\ind{x_1 \geq 0}.  
  \end{align*}
  For $x\not=0$, $x\in G$, the three points $\braces{g(x),x,f(x)}$ are
  distinct and
  \begin{displaymath}
    x = \frac12(g(x) + f(x)).   
  \end{displaymath}
  On the other hand, by the martingale property of $\gamma$ and the
  fact that $\nu(\braces{0})=0$, we have that
  \begin{displaymath}
    x = \gamma(y|x) = f(x) \gamma(y=f(x)|x) + g(x) \gamma(y=g(x)|x),
    \quad \gamma\text{-a.s.},  
  \end{displaymath}
  and therefore, the conditional probabilities
  \begin{displaymath}
    \gamma(y=f(x)|x) = \gamma(y=g(x)|x) = \frac12, \quad
    \gamma\text{-a.s.}.   
  \end{displaymath}
  For a bounded Borel function $h=h(x,y)$ on $\R^2\times \R^2$ we then
  obtain that
  \begin{align*}
    \int h(x,y) d\gamma &= \int (h(x,g(x))\ind{y=g(x)} +
                          h(x,f(x))\ind{y=f(x)}) d\gamma \\
                        &= \frac12 \int (h(x,g(x)) +
                          h(x,f(x))) d\mu. 
  \end{align*}
  Hence, $\gamma$ is unique if and only if $\mu$ is unique. We observe
  now that the map $\map{f}{G}{F_2}$ is one-to-one. Thus, for a Borel
  set $B\in \R^2$,
  \begin{align*}
    \frac12 \mu(B) = \int \ind{x\in B} \ind{y=f(x)} d\gamma =
    \int \ind{y\in f(B)} d\gamma = \nu
    (f(B)),  
  \end{align*}
  and the uniqueness of $\mu$ follows.
\end{Example}

\section{Equilibrium with insider}
\label{sec:equil-with-insid}

We consider a single-period financial market. There are a bank account
with zero interest rate and a stock. The stock value at maturity $t=1$
is represented by a random variable $V$. The stock price $S$ at
initial time $t=0$ is the result of the interaction between
\emph{noise traders}, an \emph{insider}, and \emph{market makers},
where
\begin{enumerate}
\item The noise traders place an order for $U$ stocks; $U$ is a random
  variable.
\item The insider knows the {value} of both $U$ and $V$ and places an
  order for $Q$ stocks.  The \emph{trading strategy} $Q$ is a
  $(U,V)$-measurable random variable.
\item The market makers observe only the total order $R=Q+U$.  They
  quote the price $S = f(R)$ according to a \emph{pricing rule}
  $f=f(r)$, which is a Borel function
  $\map{f}{\R}{\overline{\R}\set \R \cup
    \braces{-\infty}\cup\braces{\infty}}$.
\end{enumerate}

\begin{Definition}
  \label{def:2}
  An \emph{equilibrium} $(Q,f)$ is defined by a trading strategy $Q$
  and a pricing rule $f=f(r)$ such that
  \begin{enumerate}
  \item Given the total order $R= Q + U$, the price $S = f(R)$ is
    \emph{efficient} in the sense that
    \begin{displaymath}
      S = \cEP{V}{R}.
    \end{displaymath}
  \item Given the pricing rule $f=f(r)$, the order $Q$
    \emph{maximizes} insider's profit:
    \begin{displaymath}
      Q(V-f(Q+U))  = \max_{q\in \R} q (V-f(q+U)), 
    \end{displaymath}
    with the convention $0\times \infty = 0$.
  \end{enumerate}
\end{Definition}

\begin{Remark}
  \label{rem:4}
  Up to minor technical differences, our notion of equilibrium
  coincides with the one in~\citet{RochVila:94}. It differs from the
  classical equilibrium from \citet{Kyle:85} in the ability of the
  insider to observe noise traders' order flow $U$. In the model of
  \citet{Kyle:85}, the insider maximizes $\cEP{Q(V - f(Q+U)}{V}$ over
  all $V$-measurable random variable $Q$.
\end{Remark}

The following result links the existence of equilibrium with the
existence of an optimal map for~\eqref{eq:21} that induces an optimal
plan for~\eqref{eq:5}.

\begin{Theorem}
  \label{th:6}
  Let $Y = (U,V)\in \mathcal{L}^2$ and denote $\nu = \law{Y}$.  An
  equilibrium $(Q,f)$ exists if and only if there is an optimal map
  $X$ for~\eqref{eq:21} such that the law of $(X,Y)$ is an optimal
  plan for~\eqref{eq:5}. Insider's profit is unique and given by
  \begin{displaymath}
    Q(V-f(Q+U)) = -c(X,Y) = -\phi_G(U,V),
  \end{displaymath}
  where $G\in \mathfrak{M}$ is a maximizer for~\eqref{eq:18}.

  Moreover, there are equilibrium $(Q,f)$ and optimal map $X=(R,S)$
  such that the pricing rule $\map{f}{\R}{\overline{\R}}$ is an
  increasing function, the total order $Q+U = R$, and the price
  $f(Q+U)=S$.
\end{Theorem}

We divide the proof of the theorem into lemmas.

\begin{Lemma}
  \label{lem:12}
  Let $Y=(U,V)\in \mathcal{L}^2$, $\nu = \law{Y}$, and $X=(R,S)$ be an
  optimal map for~\eqref{eq:21} such that $S$ is $R$-measurable and
  the law of $(X,Y)$ is an optimal plan for~\eqref{eq:5}. Then there
  is an increasing function $\map{f}{\R}{\overline{\R}}$ such that
  $S=f(R)$ and $(Q,f)$ is an equilibrium with $Q=R-U$.
\end{Lemma}

\begin{proof}
  By construction, $S = \cEP{V}{R}$. Hence, we only need to verify the
  profit maximization condition for the order $Q = R-U$.
  Theorem~\ref{th:1} yields $G\in \mathfrak{M}$ such that $(R,S)\in G$
  and
  \begin{displaymath}
    (U-R)(V-S) = \min_{(r,s)\in G}(U-r)(V-s) = \phi_G(U,V). 
  \end{displaymath}
  Let $P_1$ be the projection of $G$ on the first or
  $r$-coordinate. Clearly, $P_1$ is an interval. As $S$ is
  $R$-measurable, there is an \emph{increasing} function $f=f(r)$ on
  $P_1$ such that $S=f(R)$ and $(r,f(r)) \in G$ for $r\in P_1$. By
  construction,
  \begin{displaymath}
    (U-R)(V-S) = \min_{r\in P_1} (U-r)(V-f(r)). 
  \end{displaymath}
  We now extend $f$ to an increasing function from $\R$ to
  $\overline{\R}$ by setting its values to $-\infty$ on the left and
  to $+\infty $ on the right of $P_1$. As $\phi_G(U,V)>-\infty$,
  Lemma~\ref{lem:17} yields that $U$ takes values in the closure of
  $P_1$. Under the standing convention: $0\times \infty = 0$, we
  obtain that
  \begin{displaymath}
    \phi_G(U,V)= (U-R)(V-S) = \min_{r\in \R}(U-r)(V-f(r)). 
  \end{displaymath}
  Hence, $(Q,f)$ is an equilibrium with $Q=R-U$.
\end{proof}

\begin{Lemma}
  \label{lem:13}
  Let $\map{f}{\R}{\overline{\R}}$ be a Borel function and
  \begin{displaymath}
    \phi(y) = \inf_{r\in \R}(y_1 - r)(y_2 - f(r)) \in [-\infty,0],
    \quad y\in \R^2,  
  \end{displaymath}
  with the convention: $0\times \infty=0$.  Then there is
  $G\in \mathfrak{M}$ such that $\phi\leq \phi_G$.
\end{Lemma}

 \begin{proof}
   Given $y^0,y^1\in \R^2$ and $t\in [0,1]$, we denote
   $r = y^0_1 + t(y^1_1 - y^0_1)$ and deduce that
   \begin{align*}
     (1-t) \phi(y^0) + t \phi(y^1) & \leq (1-t) \min((y^0_1 - r)(y^0_2
                                     -
                                     f(r)),0) \\
                                   &\quad  + t \min((y^1_1 - r)(y^1_2 - f(r)),0) \\
                                   &\leq t(1-t) (y^1_1-y^0_1)(y^1_2-y^0_2),
   \end{align*}
   where in the middle we used the negative parts to account for the
   possibility that $\abs{f(r)}=\infty$.  The result now follows
   from~Lemma~\ref{lem:9}.
 \end{proof}

 \begin{Lemma}
   \label{lem:14}
   Let $Y=(U,V)\in \mathcal{L}^2$, $\nu = \law{Y}$, and $(Q,f)$ be an
   equilibrium with the total order $R=Q+U$ and the price
   $S=f(R)$. Then $X=(\widetilde R, S)$ with $\widetilde R=\cEP{U}{R}$
   is an optimal map for~\eqref{eq:21}, the law of $(X,Y)$ is an
   optimal plan for~\eqref{eq:5}, and
   \begin{equation}
     \label{eq:30}
     Q(f(Q+U) - V) = (R-U)(S-V) = (\widetilde R-U)(S-V). 
   \end{equation}
 \end{Lemma}

 \begin{proof}
   From the definition of the equilibrium we obtain that
   \begin{displaymath}
     \phi(U,V) =  Q(f(Q+U)-V)= (R-U)(S-V), 
   \end{displaymath}
   where $\phi(u,v) = \inf_{r\in \R}(u-r)(v-f(r))$, $(u,v)\in \R^2$.
   We claim that
   \begin{equation}
     \label{eq:31}
     \EP{\phi(U,V)}  = \EP{(U-\widetilde R)(V-S)}.
   \end{equation}
   As the integrability properties of $R$ are unknown, we use a
   localization argument.  For $n\geq 1$ from the martingale
   properties $\cEP{V}{R} = S$ and $\cEP{U}{R} = \widetilde R$ we
   deduce that
   \begin{align*}
     \EP{\phi(U,V) \ind{\abs{R}\leq n}} & =
                                          \EP{(U-R)(V-S)\ind{\abs{R}\leq
                                          n}} \\
                                        & = \EP{U(V-S)\ind{\abs{R}\leq
                                          n}} \\
                                        & = \EP{(U-\widetilde R)(V-S)\ind{\abs{R}\leq n}}.
   \end{align*}
   Taking the limit as $n\to \infty$, we obtain~\eqref{eq:31} by the
   dominated convergence theorem.

   Lemma~\ref{lem:13} yields $G\in \mathfrak{M}$ such that
   $\phi \leq \phi_G$. From Lemma~\ref{lem:7} we deduce that
   \begin{displaymath}
     \EP{\phi(U,V)} \leq \EP{\phi_G(U,V)} = \int \phi_G d\nu \leq \int
     c(x,y) d\gamma, \quad 
     \gamma\in \Gamma(\nu). 
   \end{displaymath}
   In view of~\eqref{eq:31} and since
   $\widetilde\gamma = \law{\widetilde R,S,U,V}$ belongs to
   $\Gamma(\nu)$, we obtain that
   \begin{displaymath}
     \EP{(U-\widetilde R)(V-S)} = \EP{\phi(U,V)} = \EP{\phi_G(U,V)} =  \int
     c(x,y) d\widetilde \gamma.
   \end{displaymath}
   It follows that $\widetilde\gamma$ is an optimal plan,
   $(\widetilde R, S)$ is an optimal map, and
   $\phi(U,V) = \phi_G(U,V)$. Finally, Theorem~\ref{th:2} yields that
   \begin{displaymath}
     \phi_G(U,V)= c(X,Y) = (\widetilde R-U)(S-V), 
   \end{displaymath}
   and we obtain~\eqref{eq:30}.
 \end{proof}

 \begin{proof}[Proof of Theorem~\ref{th:6}] If $X=(R,S)$ is an optimal
   map, then $\widetilde X = (R,\widetilde S)$ with
   $\widetilde S = \cEP{V}{R} = \cEP{S}{R}$ is an optimal map as well
   and $\widetilde S$ is $R$-measurable.  By Theorem~\ref{th:2},
   \begin{displaymath}
     c(X,Y) = c(\widetilde X,Y) = \phi_G(Y) = \phi_G(U,V)
   \end{displaymath}
   for every maximizer $G\in \mathfrak{M}$ to~\eqref{eq:18}. In
   particular, $c(X,Y)$ is the same random variable for every optimal
   map $X$.  After these observations, the proof follows from
   Lemmas~\ref{lem:12} and~\ref{lem:14}.
 \end{proof}

 We now state sufficient conditions for the existence and uniqueness
 of equilibrium. Theorem~\ref{th:7} generalizes a result
 from~\citet{RochVila:94}, where the distribution of $(U,V)$ has a
 compact support and a continuous density.

 \begin{Theorem}
   \label{th:7}
   Let $Y=(U,V)\in \mathcal{L}^2$ and suppose that the law of ${Y}$ is
   $\mathcal{D}$-regular. Then an equilibrium $(Q,f)$ exists.

   If, in addition, the laws of $U$ and $V$ are atomless, then
   insider's order $Q$, the total order $R=Q+U$, and the price
   $S=f(R)$ are unique. Moreover, $X=(R,S)$ is the unique optimal map
   for~\eqref{eq:21} and $\gamma = \law{X,Y}$ is the unique optimal
   plan for~\eqref{eq:5}.
 \end{Theorem}

 For the proof we need a lemma.

 \begin{Lemma}
   \label{lem:15}
   Let $\map{f}{\R}{\overline{\R}}$ be a Borel function and
   \begin{displaymath}
     \phi(u,v) = \inf_{r\in \R}(u - r)(v - f(r)) \in [-\infty,0],
     \quad (u,v)\in \R^2,  
   \end{displaymath}
   with the convention: $0\times \infty=0$. There is a countable set
   $A\subset \R$ such that if $u,v\notin A$ and $\phi(u,v)=0$, then
   \begin{displaymath}
     f^{-1}(v)
     \set \descr{r\in \R}{f(r) = v} = \braces{u}. 
   \end{displaymath}
 \end{Lemma}

\begin{proof}
  If $\phi(u,v) = \inf_{r\in \R}(u - r)(v - f(r)) = 0$, then
  \begin{displaymath}
    g(u) \set  \sup_{r<u} f(r) \leq v \leq \inf_{r>u} f(r) \set h(u).  
  \end{displaymath}
  Clearly,
  \begin{displaymath}
    \inf_{r\geq u} f(r) \leq f(u) \leq \sup_{r\leq u} f(r).
  \end{displaymath}
  Thus, if the increasing functions $g$ and $h$ are continuous and
  strictly increasing at $u$, then $f^{-1}(v) = \braces{u}$. To
  conclude the proof we just observe that the set of arguments, where
  an increasing function is discontinuous, and the set of values,
  where it is not strictly increasing, are countable.
\end{proof}

\begin{proof}[Proof of Theorem~\ref{th:7}]
  If $\nu = \law{Y}$ is $\mathcal{D}$-regular, then Theorem~\ref{th:4}
  yields an optimal map $X$ such that the law of $(X,Y)$ is an optimal
  plan. By Theorem~\ref{th:6}, there is an equilibrium $(Q,f)$.

  If the laws of $U = Y_1$ and $V=Y_2$ are atomless, then, by
  Theorem~\ref{th:5}, the optimal map $X=(X_1,X_2)$ is
  unique. Lemma~\ref{lem:14} shows that $S=X_2$ is the unique
  equilibrium price: $S=f(R)$, where $R=Q+U$.

  Let $\phi$ be the function defined in Lemma~\ref{lem:15} and
  $G\in \mathfrak{M}$ be a maximizer for~\eqref{eq:18}. From the
  definition of the equilibrium and Theorem~\ref{th:6} we deduce that
  \begin{displaymath}
    (U-R)(V-f(R)) = (U-R)(V-S) = \phi(U,V) = \phi_G(U,V). 
  \end{displaymath}
  If $\phi(U,V)<0$, then the total order $R$ is clearly unique. If
  $\phi(U,V)=0$, then $R=U$ by Lemma~\ref{lem:15} and the continuity
  of the distributions of $U$ and $V$. Thus, the total order $R$ and
  insider's order $Q = R-U$ are unique.  By Theorem~\ref{th:6}, the
  uniqueness of $R$ and $S$ implies that $(R,S)$ is an optimal
  map. Hence $X_1 = R$.
\end{proof}

\appendix
\section{Closure of probability measures with bounded densities in
  $\cW_p(\R^d)$}
\label{sec:density-cw_2rd}

Let $p\geq 1$. We denote by $\cW_p(\R^d)$ the space of Borel
probability measures on $\R^d$ with finite $p$-th moments equipped
with the Wasserstein metric:
\begin{displaymath}
  W_p(\mu,\nu)=\left\lbrace \inf_{\gamma\in\Pi(\mu,\nu)}
    \int |x-y|^p d\gamma  \right\rbrace^{1/p}, 
\end{displaymath}
where $\Pi(\mu,\nu)$ is the family of Borel probability measures
$\gamma$ on $\R^d\times \R^d = \descr{(x,y)}{x,y \in \R^d}$ with
$x$-marginal $\mu$ and $y$-marginal $\nu$. We recall that
$\mathcal{W}_p(\R^d)$ is a complete separable metric space and that
$\mu_n \rightarrow \mu$ in $\cW_p(\R^d)$ if and only if
$\int f(x) d\mu_n \rightarrow \int f(x) d\mu$ for every continuous
function $f$ with polynomial $p$-th growth:
\begin{displaymath}
  |f(x)|\leq K(1+|x|^p), \quad x\in \R^d. 
\end{displaymath}

Let $\nu\in \mathcal{W}_p(\R^d)$ and $\mathcal{Q}_\infty(\nu)$ be the
family of Borel probability measures on $\R^d$ that have bounded
densities with respect to $\nu$:
\begin{displaymath}
  \mathcal{Q}_\infty(\nu) \set \descr{\mu \ll
    \nu}{\frac{d\mu}{d\nu} \in \mathcal{L}^\infty(\R^d)}. 
\end{displaymath}
Clearly, $\mathcal{Q}_\infty(\nu) \subset \mathcal{W}_p(\R^d)$. The
following result, used in the proof of our main Theorem~\ref{th:1},
describes the closure of $\mathcal{Q}_\infty(\nu)$ under $W_p$.

\begin{Theorem}
  \label{th:8}
  Let $p\geq 1$ and $\nu\in \mathcal{W}_p(\R^d)$. Then the closure of
  $\mathcal{Q}_\infty(\nu)$ in $\cW_p(\R^d)$ has the form:
  \begin{displaymath}
    \mathcal{S}_p(\nu) = 
    \descr{\mu\in \mathcal{W}_p(\R^d)}{\supp{\mu}\subset
      \supp{\nu}}. 
  \end{displaymath}
\end{Theorem}
\begin{proof}
  If $\mu_n\to \mu$ in $\mathcal{W}_p(\R^d)$, then $\mu_n\to \mu$
  weakly and thus,
  \begin{displaymath}
    \mu(C) \geq \limsup_{n\to \infty} \mu_n(C),
  \end{displaymath}
  for every closed set $C$. In particular, if
  $(\mu_n)\subset \mathcal{S}_p(\nu)$, then
  \begin{displaymath}
    \mu(\supp{\nu}) \geq  \limsup_{n\to \infty} \mu_n(\supp{\nu})= 1
  \end{displaymath}
  and hence, $\mu \in \mathcal{S}_p(\nu)$. It follows that
  $\mathcal{S}_p(\nu)$ is closed in $\mathcal{W}_p(\R^d)$. Clearly,
  $\mathcal{S}_p(\nu)$ is convex.

  If $\supp{\nu}$ is compact, then restricted to $\mathcal{S}_p(\nu)$
  the convergence under $W_p$ is equivalent to the weak convergence
  and thus, the family of probability measures in $\mathcal{S}_p(\nu)$
  with \emph{finite} support is dense. Being a closed convex set,
  $\mathcal{S}_p(\nu)$ is then the closure of
  $\mathcal{Q}_\infty(\nu)$ in $\mathcal{W}_p(\R^d)$ if and only if
  every Dirac measure $\delta_y=\delta_y(dx)$ concentrated at
  $y\in \supp{\nu}$ is the weak limit of a sequence
  $(\mu_n) \subset \mathcal{Q}_\infty(\nu)$.  The sequence $(\mu_n)$
  with
  \begin{displaymath}
    \frac{d\mu_n}{d\nu}(x) = \frac1{\nu(B_{1/n} (y))}
    \ind{x\in B_{1/n}(y)},\quad n\geq 1, 
  \end{displaymath}
  where $B_r(y)$ is the ball of radius $r>0$ centered at $y$, has the
  required properties.

  If $\supp{\nu}$ is not compact, we approximate
  $\mu \in \mathcal{S}_p(\nu)$ by the sequence $(\mu_n)$ given by
  \begin{displaymath}
    \frac{d\mu_n}{d\mu}(x) = \frac1{\mu(B_n(y))} \ind{x\in B_n(y)}, \quad
    n\geq 1, 
  \end{displaymath}
  for some $y\in \supp{\mu}$. We have that
  $(\mu_n)\subset \mathcal{S}_p(\nu)$ and $\mu_n \to \mu$ under
  $W_p$. By what we have already proved, each $\mu_n$ belongs to the
  closure of $\mathcal{Q}_\infty(\nu_n)$ in $\mathcal{W}_p(\R^d)$,
  where
  \begin{displaymath}
    \frac{d\nu_n}{d\nu}(x) = \frac1{\nu(B_n(y))} \ind{x\in B_n(y)}, \quad
    n\geq 1.
  \end{displaymath}
  As $\mathcal{Q}_\infty(\nu_n) \subset \mathcal{Q}_\infty(\nu)$,
  $n\geq 1$, we deduce that the sequence $(\mu_n)$ belongs to the
  closure of $\mathcal{Q}_\infty(\nu)$ in $\mathcal{W}_p(\R^d)$. Same
  property holds for its $W_p$-limit $\mu$ and the result follows.
\end{proof}

\section{Properties of the function $\phi_G$}
\label{sec:dual-space-functions}

Let $G$ be a maximal monotone set: $G\in \mathfrak{M}$. In this
appendix, we collect the properties of the function
\begin{displaymath}
  \phi = \phi_G(y) \set \inf_{x\in G} c(x,y) = \inf_{x\in G}
  (x_1-y_1)(x_2-y_2), \quad y\in \R^2, 
\end{displaymath}
used throughout the paper.

\begin{Lemma}
  \label{lem:16}
  The function $\phi = \phi_G$ and its $c$-conjugate
  \begin{displaymath}
    \phi^c(x) = \inf_{y\in \R^2}(c(x,y) - \phi(y)), \quad x\in \R^2, 
  \end{displaymath}
  take values in $[-\infty,0]$, $\phi^c\leq \phi$, and
  \begin{displaymath}
    G = \descr{y\in \R^2}{\phi(y) = 0} = \descr{x\in \R^2}{\phi^c(x) = 0}. 
  \end{displaymath}
\end{Lemma}
\begin{proof}
  Let $y\in G$. As $G\in \mathfrak{M}$, we have that $c(x,y)\geq 0$,
  $x\in G$, and thus, $\phi(y)=0$.  Conversely, if $y\not\in G$, then
  the maximal monotone set $G$ crosses the interior of either the
  upper-left or the lower-right quadrants relative to $y$. If $z\in G$
  belongs to such intersection, then
  \begin{displaymath}
    \phi(y) \leq c(z,y) = (z_1-y_1)(z_2-y_2) < 0. 
  \end{displaymath}

  We have shown that $\phi<0$ on $\R^2\setminus G$ and $\phi=0$ on
  $G$. It follows that
  \begin{displaymath}
    \phi^c(x) \leq \inf_{y\in G}(c(x,y) - \phi(y)) = \inf_{y\in G}
    c(x,y) = \phi(x)\leq 0.  
  \end{displaymath}
  If $\phi^c(x)=0$, then $\phi(x)=0$ and thus, $x\in G$. Conversely,
  if $x\in G$, then $c(x,y)-\phi(y)\geq 0$, $y\in \R^2$, and
  therefore, $\phi^c(x) = 0$.
\end{proof}

We associate with $\phi$ the closed convex function
\begin{displaymath}
  \psi(y) = \psi_G(y) = y_1y_2 - \phi(y) = \sup_{x\in G}(x_1y_2 + x_2y_1 -
  x_1x_2), \quad y\in \R^2.  
\end{displaymath}
Clearly, $\phi$ and $\psi$ have same domains:
\begin{displaymath}
  \dom{\phi} = \descr{y\in \R^2}{\phi(y)>-\infty} = \descr{y\in
    \R^2}{\psi(y)<\infty} = \dom{\psi}. 
\end{displaymath}
For a convex set $A\subset \R^d$ we denote by $\closure{A}$,
$\interior{A}$, $\ri{A}$, and
$\boundary{A} = \closure{A}\setminus \ri{A}$ its respective closure,
interior, relative interior and relative boundary.

\begin{Lemma}
  \label{lem:17}
  The domain of $\phi$ is convex. If $G$ is either horizontal or
  vertical line, then $\dom{\phi} = G$. Otherwise, $\dom{\phi}$ has a
  non-empty interior:
  \begin{equation}
    \label{eq:32}
    \interior{\dom{\phi}} = \interior{P_1}\times \interior{P_2},
  \end{equation}
  where $P_i$ is the projection of $G$ on $x_i$-coordinate,
  $i=1,2$. If $y\in \boundary{\dom{\phi}}\cap \dom{\phi}$, then the
  relative interiors of the horizontal and vertical parts of
  $\boundary{\dom{\phi}}$ containing $y$ also belong to $\dom{\phi}$.
\end{Lemma}

\begin{proof}
  Being convex, the function $\psi=\psi_G$ has convex domain. As
  $\dom{\phi} = \dom{\psi}$, the domain of $\phi$ is also convex.

  We observe that $P_i$ is either a point or an interval. If
  $P_1 = \braces{a_1}$, then $G$ is a vertical line:
  $G = \descr{x\in \R^2}{x_1 = a_1}$. For $y\not \in G$ we have that
  $\abs{y_1-a_1}>0$ and
  \begin{displaymath}
    \phi(y) = \inf_{x_2 \in \R} (y_1-a_1)(y_2-x_2) = -\infty. 
  \end{displaymath}
  Thus, $\dom \phi = G$. The case where $P_2$ is a point and thus, $G$
  is a horizontal line is identical.

  We assume now that $\interior{P_i} = (a_i, b_i)$, where
  $-\infty \leq a_i < b_i \leq \infty$. If
  $y=(y_1,y_2)\in (a_1,b_1)\times (a_2,b_2)$, then the set
  $C\set\descr{x\in G}{c(x,y) \leq 0}$ is bounded and therefore,
  \begin{displaymath}
    \phi(y) = \inf_{x\in G} c(x,y) = \inf_{x\in C} c(x,y) >
    -\infty.   
  \end{displaymath}
  Conversely, suppose that $y$ does not belong to the closure of
  $P_1\times P_2$, say $y_1 < a_1$; other cases are covered
  similarly. Then $a_2= -\infty$ and hence,
  \begin{displaymath}
    \phi(y) = \inf_{x\in G} c(x,y) \leq (a_1 - y_1 ) (a_2-y_2) =
    -\infty. 
  \end{displaymath}
  We have proved~\eqref{eq:32}.

  For the last assertion of the lemma, we assume that $a_1>-\infty$
  and take $y=(a_1,y_2)$ and $z=(a_1,z_2)$ with $z_2<b_2$.  Given that
  $\phi(y)>-\infty$, we have to show that $\phi(z)>-\infty$. Indeed,
  otherwise there is a sequence $(x^n) \subset G$ such that
  \begin{displaymath}
    \lim_{n\rightarrow \infty} (a_1 - x_1^n) (z_2 - x_2^n)  = -\infty. 
  \end{displaymath}	  
  Since $z_2 < b_2$, the sequence $(x_1^n)$ is bounded and
  $x_2^n \rightarrow a_2 = -\infty$. It follows that
  \begin{displaymath}
    \phi(y) \leq \limsup_{n\rightarrow \infty} (a_1 - x_1^n)(y_2 -
    z_2)  + \lim_{n\rightarrow \infty}(a_1 - x_1^n)(z_2 - x_2^n) 
    = -\infty,
  \end{displaymath}
  and we obtain a contradiction.
\end{proof}

The closed convex function $\psi=\psi_G$ is lower semi-continuous on
$\R^2$ and is continuous on the interior of its domain.  The following
result shows that $\phi$ and $\psi$ are continuous relative to their
full domains.

\begin{Lemma}
  \label{lem:18}
  If $(y^n)\subset \dom{\psi}=\dom{\phi}$ and $y^n \to y$, then
  $\psi(y^n)\to \psi(y)$ and $\phi(y^n) \to \phi(y)$.
\end{Lemma}
\begin{proof}
  It is sufficient to consider the case of the function $\psi$ and
  take $y\in \boundary{\dom{\psi}}$. If $\psi(y)=\infty$, then the
  result holds by the lower semi-continuity:
  \begin{displaymath}
    \liminf_{n\to \infty} \psi(y^n) \geq \psi(y) = \infty. 
  \end{displaymath}
  Thus, we assume that $\psi(y)<\infty$ or, equivalently, that
  $y\in \boundary{\dom{\psi}} \cap \dom{\psi}$. By Lemma~\ref{lem:17},
  the relative interiors of the horizontal and vertical parts of
  $\partial \dom{\psi}$ containing $y$ belong to $\dom{\psi}$. Hence,
  there is a closed triangle in $\dom{\psi}$ that contains
  $(y^n)_{n\geq n_0}$, for sufficiently large $n_0$. Being convex, the
  function $\psi$ is continuous on this triangle and the result
  follows.
\end{proof}

We define a multi-valued function
\begin{align*}
  {\Arg}_G(y) \set \argmin_{x\in G} c(x,y) =\descr{x\in G}{\phi(y) =
  c(x,y)}, \quad y\in \R^2,
\end{align*}
taking values in the closed (possibly empty) subsets of $G$, and
denote
\begin{align*}
  \dom{{\Arg}_G} \set \descr{y\in \R^2}{{\Arg}_G(y) \not=\emptyset}.
\end{align*}
Let $E^G_i = \cup_{t\in \mathcal{T}^G_i}E^G_i(t)$, $i=1,2$, be the
union of vertical and horizontal line segments of $G$;
see~\eqref{eq:23}. As the set $G$ is fixed, we write simply
\begin{displaymath}
  \Arg = {\Arg}_G,\quad  E_i = E^G_i, \quad E_i(t) = E^G_i(t), \quad
  \mathcal{T}_i = \mathcal{T}^G_i.
\end{displaymath}
The following lemma shows that for $y\in \dom{\Arg}\setminus G$ the
set $\Arg(y)$ can intersect $E_i(t)$ only at $\boundary{E_i(t)}$.  We
denote $\ip{x}{y} \set \sum_{i=1}^2 x_iy_i$, the scalar product of
$x,y\in \R^2$.

\begin{Lemma}
  \label{lem:19}
  Let $i\in \braces{1,2}$ and $t\in \mathcal{T}_i$. If
  $y\in \dom{\Arg}\setminus G$ and $x\in E_i(t)\cap \Arg(y)$, then $x$
  belongs to the boundary of ${E_i(t)}$ and
  \begin{displaymath}
    \ip{z-x}{y-x}> 0, \quad z\in \ri{E_i(t)}.   
  \end{displaymath}
\end{Lemma}

\begin{proof}
  Without loss of generality we can assume that $i=2$ and that $y$
  stays above $G$. Then the increasing hyperbola
  \begin{displaymath}
    H = \descr{z\in \R^2}{c(z,y) = \phi(y), \; z_1>y_1}
  \end{displaymath}
  contains $x$ and lays below $G$, which is only possible if $x$ is
  the right boundary of the horizontal line segment $E_2(t)$. In this
  case,
  \begin{displaymath}
    \ip{z-x}{y-x} = (z_1-x_1)(y_1-x_1) >0, \quad
    z\in \ri{E_2(t)}, 
  \end{displaymath}
  and the result follows.
\end{proof}

For $x,y\in \R^2$ we denote by $L(x,y)$ the line segment connecting
$x$ and $y$:
\begin{displaymath}
  L(x,y) \set \descr{tx+(1-t)y}{t\in [0,1]}. 
\end{displaymath}

\begin{Lemma}
  \label{lem:20}
  Let $y^0$ and $y^1$ belong to $\dom{\Arg}\setminus G$ and stay above
  and below $G$, respectively. If $x\in \Arg(y^0)\cap \Arg(y^1)$, then
  $x$ belongs to the line segment $L(y^0,y^1)$ connecting $y^0$ and
  $y^1$.
\end{Lemma}

\begin{proof}
  The conditions of the lemma imply that the increasing hyperbolas
  \begin{align*}
    H^0 &= \descr{z\in \R^2}{c(y^0,z)=\phi(y^0), \; z_1 > y^0_1}, \\
    H^1 &= \descr{z\in \R^2}{c(y^1,z)=\phi(y^1), \; z_1 < y^1_1},
  \end{align*}
  contain $x$ and stay below and above $G$, respectively. Hence, they
  have identical tangent lines at $x$.  Elementary computations show
  that the slope of the tangent line is given by
  \begin{displaymath}
    \frac{x_2 - y^0_2}{y^0_1-x_1} = \frac{x_2 - y^1_2}{y^1_1 - x_1}  
  \end{displaymath}
  and the result follows.
\end{proof}

For $y\in \interior{\dom{\phi}}$ the derivative $\nabla \phi(y)$ is
defined in the classical sense. For $y\in \boundary{\dom{\phi}}$ the
derivative $\nabla \phi(y)$ exists if it is the limit:
$\nabla \phi(y^n)\to \nabla \phi(y)$, for every sequence
$(y^n)\subset \interior{\dom{\phi}}\cap \dom{\nabla \phi}$ that
converges to $y$.  We write
\begin{align*}
  \dom{\nabla \phi} \set \descr{y\in \dom{\phi}}{\nabla
  \phi(y)\;\text{exists}}.
\end{align*}
By $D^c \set (D^c_1, D^c_2)$ we denote the differential operator
associated with the cost function $c=c(x,y)$:
\begin{align*}
  D^c_1\phi(y) \set y_1 - \frac{\partial \phi}{\partial y_2}(y), \quad
  D^c_2\phi(y) \set y_2 - \frac{\partial \phi}{\partial y_1}(y).
\end{align*}
Finally, let $E=E^G = E^G_1 \cup E^G_2$ be the union of the vertical
and horizontal line segments of $G$ and denote
\begin{align*}
  \udom{{\Arg}} \set \descr{y\in \R^2}{{\Arg}(y) \text{ is a
  singleton}}.
\end{align*}
We observe that
\begin{equation}
  \label{eq:33}
  E = G \setminus \udom{\Arg}. 
\end{equation}

\begin{Theorem}
  \label{th:9}
  We have that
  \begin{equation}
    \label{eq:34}
    \dom{\nabla \phi} \setminus \udom{\Arg} \subset
    \boundary{\dom{\phi}} \cap G = 
    \boundary{\dom{\phi}} \cap E.
  \end{equation}
  Conversely, the set difference
  $\udom{\Arg}\setminus \dom{\nabla \phi}$ has at most two points and
  these points belong to different linear parts of
  $\boundary{\dom{\phi}}$. If
  $y\in \dom{\nabla \phi} \cap \udom{\Arg}$, then $D^c\phi(y)$ is the
  only element of $\Arg(y)$ and
  \begin{equation}
    \label{eq:35}
    \phi(y) = c(D^c\phi(y),y) = \frac{\partial \phi}{\partial x_1}(y)
    \frac{\partial \phi}{\partial x_2}(y).
  \end{equation}
\end{Theorem}

We divide the proof of the theorem into lemmas.  We write $x\leq y$ if
$x_i\leq y_i$, $i=1,2$.  If $x,y\in G$ and $x\leq y$, then $G(x,y)$
denotes the segment of $G$ bounded by $x$ and $y$:
\begin{displaymath}
  G(x,y) \set \descr{z\in G}{x\leq z\leq y}. 
\end{displaymath}

\begin{Lemma}
  \label{lem:21}
  Let $y\in \interior{\dom{\phi}}$. Then $y\in \dom{\nabla \phi}$ if
  and only if $y\in \udom{\Arg}$. In this case, $D^c\phi(y)$ is the
  only element of $\Arg(y)$.
\end{Lemma}

\begin{proof}
  From the structure of $\interior{\dom{\phi}}$ in Lemma~\ref{lem:17}
  we deduce the existence of $x^0,x^1\in G$ such that $x^0\leq x^1$,
  $y\in \interior{R(x^0,x^1)}$, and
  $R(x^0,x^1)\subset \interior{\dom{\phi}}$, where $R(x^0,x^1)$ is the
  rectangle with the diagonal $L(x^0,x^1)$.  Every $x\in G$ such that
  $c(x,y)\leq 0$ belongs to $G(x^0,x^1) = R(x^0,x^1) \cap G$. Hence,
  \begin{displaymath}
    \Arg(y) = \argmin_{z\in G(x^0,x^1)}c(z,y)
    = \argmax_{z\in G(x^0,x^1)}(y_1z_2+y_2z_1-z_1z_2).  
  \end{displaymath}
  As $G(x^0,x^1)$ is compact, $\Arg(y)$ is non-empty. If
  $x\in \Arg(y)$, then
  \begin{align*}
    \psi(u) - \psi(y) &= \sup_{z\in G}(u_1z_2+u_2z_1-z_1z_2) -
                        (y_1x_2+y_2x_1-x_1x_2) \\
                      &\geq x_2(u_1-y_1) + x_1(u_2-y_2), \quad u\in \R^2.
  \end{align*}
  It follows that $(x_2,x_1)$ belongs to $\partial \psi(y)$, the
  subdifferential of $\psi = \psi_G$ at $y$. Differentiability of
  $\phi$ (equivalently, of $\psi$) at $y$ then implies that $\Arg(y)$
  is a singleton and
  \begin{displaymath}
    x_1 = \frac{\partial \psi}{\partial y_2}(y) = y_1 - \frac{\partial
      \phi}{\partial y_2}(y) = D^c_1\phi(y), \quad x_2 =
    D^c_2\phi(y). 
  \end{displaymath}

  Conversely, let $x$ be the only element of $\Arg(y)$ and
  $\widetilde x = (\widetilde x_1, \widetilde x_2) \in \R^2$ be such
  that $(\widetilde x_2, \widetilde x_1) \in \partial \psi(y)$. We
  have to show that $x = \widetilde x$. We take a unit vector
  $e=(e_1,e_2)$ in $\R^2$ and define a sequence $(y^n)$ in $\R^2$ such
  that
  \begin{displaymath}
    y^n_1 = y_1 + \frac1n e_2, \; y^n_2 = y_2 + \frac1n e_1, \quad
    n\geq 1.  
  \end{displaymath}
  Let $n_0$ be an index such that $y^n \in \interior{R(x^0,x^1)}$,
  $n\geq n_0$.  By the first part of the proof, for $n\geq n_0$ the
  set $\Arg(y^n)$ is non-empty and belongs to the compact
  $G(x^0,x^1)$. Moreover, if $x^n \in \Arg(y^n)$ then
  $(x^n_2,x^n_1) \in \partial \psi(y^n)$.  It follows that
  \begin{align*}
    x^n_2(y^n_1-y_1) + x^n_1(y^n_2-y_2)
    &\geq \psi(y^n) - \psi(y) \\
    &\geq \widetilde x_2(y^n_1-y_1) + \widetilde x_1(y^n_2-y_2),
  \end{align*}
  and then that $\ip{x^n}{e} \geq \ip{\widetilde x}{e}$.  As $x$ is
  the only element of $\Arg(y)$ and
  \begin{displaymath}
    \psi(y^n) =  y^n_1x^n_2+y^n_2x^n_1-x^n_1x^n_2 \to \psi(y) =
    y_1x_2+y_2x_1-x_1x_2, 
  \end{displaymath}
  every convergent subsequence of $(x^n)$ goes to $x$ and then
  $x^n\to x$.  Hence, $\ip{x}{e} \geq \ip{\widetilde x}{e}$ and, as
  $e$ is an arbitrary unit vector in $\R^2$, we obtain that
  $x = \widetilde x$.
\end{proof}

\begin{Lemma}
  \label{lem:22}
  Let $y\in \dom{\Arg}\setminus G$ and $x\in \Arg(y)$. Then
  \begin{displaymath}
    \ri{L(x,y)} = \descr{ty + (1-t)x}{t\in (0,1)} \subset \dom{\nabla\phi}
  \end{displaymath}
  and $D^c\phi(z) = x$, $z\in \ri{L(x,y)}$. The slope of the line
  segment $L(x,y)$ is negative and has the form:
  \begin{displaymath}
    \frac{y_2-x_2}{y_1-x_1} =  \frac{z_2-x_2}{z_1-x_1} =
    \frac{\left(\frac{\partial \phi}{\partial 
          y_1}(z)\right)^2}{\phi(z)}
    = \frac{\phi(z)}{\left(\frac{\partial \phi}{\partial
          y_2}(z)\right)^2}, \quad z\in \ri{L(x,y)}. 
  \end{displaymath}
\end{Lemma}

\begin{proof}
  We fix $t\in (0,1)$ and denote $y(t) = ty + (1-t)x$. From the
  description of $\interior{\dom{\phi}}$ in Lemma~\ref{lem:17} we
  deduce that $y(t)\in \interior{\dom{\phi}}$.  Without loss in
  generality we can assume that $y_1<x_1$. Then the hyperbola
  \begin{displaymath}
    H = \descr{z\in \R^2}{c(z,y) = \phi(y), \; z_1>y_1}
  \end{displaymath}
  contains $x$ and stays below $G$, while the hyperbola
  \begin{displaymath}
    H(t) = \descr{z\in \R^2}{c(z,y(t)) = c(x, y(t)), \; z_1>y_1(t)}
  \end{displaymath}
  contains $x$ and stays below $H$. It follows that $x$ is the only
  element of $\Arg(y(t))$. Lemma~\ref{lem:21} yields that
  $D^c\phi(y(t)) =x$. The last part of the lemma follows directly from
  the definition of $D^c\phi$ and the fact that
  $\phi(y(t)) = c(x,y(t)) = c(D^c\phi(y(t)),y(t))$.
\end{proof}

The following corollary of Lemma~\ref{lem:22} will also be used in the
proof of Theorem~\ref{th:10}.

\begin{Lemma}
  \label{lem:23}
  Let $y^0$ and $y^1$ be distinct points in $\dom{\Arg}\setminus G$
  and $x^i \in \Arg(y^i)$, $i=1,2$. Then either $x^0=x^1$ or the line
  segments $L(x^0,y^0)$ and $L(x^1,y^1)$ do not intersect.
\end{Lemma}

\begin{proof}
  If $L(x^0,y^0)$ and $L(x^1,y^1)$ have common \emph{interior} point
  $z$, then Lemma \ref{lem:22} yields that $x^0= D^c\phi(z) = x^1$.
\end{proof}

\begin{Lemma}
  \label{lem:24}
  Let $y\in \dom{\nabla \phi}\setminus G$. Then $D^c\phi(y)$ is the
  only element of $\Arg(y)$.
\end{Lemma}

\begin{proof}
  In view of Lemma~\ref{lem:21} we can further assume that
  $y\in \boundary{\dom{\phi}}$. Let $(y^n)$ be a sequence in
  $\interior{\dom{\phi}}\cap \dom{\nabla \phi}$ that converges to
  $y$. By Lemma~\ref{lem:21}, $D^c\phi(y^n)$ is the only element of
  $\Arg(y^n)$.  From the construction of $\nabla\phi$ on
  $\boundary{\dom{\phi}}$ and Lemma~\ref{lem:18} we deduce that
  \begin{align*}
    D^c\phi(y) &= \lim_{n\to \infty} D^c \phi(y^n) \in G, \\
    \phi(y) &= \lim_{n\to \infty} \phi(y^n) = \lim_{n\to \infty}
              c(D^c\phi(y^n),y^n) = c(D^c\phi(y),y).
  \end{align*}
  Hence, $D^c\phi(y) \in \Arg(y)$.  On the other hand, if
  $x\in \Arg(y)$, then Lemma~\ref{lem:22} allows us to choose the
  sequence $(y^n)$ so that $D^c\phi(y^n) = x$. Hence, $x= D^c\phi(y)$.
\end{proof}

\begin{Lemma}
  \label{lem:25}
  The set difference $\udom{\Arg}\setminus\dom{\nabla \phi}$ has at
  most two points and these points belong to different linear parts of
  $\boundary{\dom{\phi}}$.
\end{Lemma}

\begin{proof}
  From Lemma~\ref{lem:17} we deduce that
  $\interior{\dom{\phi}} = (a_1,b_1)\times (a_2,b_2)$, where
  $-\infty \leq a_i < b_i \leq \infty$ and $(a_i,b_i)$ is the interior
  of the projection of $G$ on the $x_i$-coordinate. Without loss of
  generality we can assume that $a_1>-\infty$.  Let $y^0$ and $y^1$ be
  such that $y^0_1 = y^1_1 = a_1$, $y^0_2 < y^1_2<b_2$ and
  $y^0\in \dom{\Arg}$, $y^1\in \udom{\Arg}$. We are going to show that
  $y^1\in \dom{\nabla \phi}$.  By doing so, we shall prove that the
  interior of each linear part of $\boundary{\dom{\phi}}$ has at most
  one element of $\udom{\Arg}/\dom{\nabla \phi}$.

  Let $(z^n)$ be a sequence in
  $\interior{\dom{\phi}}\cap \dom{\nabla\phi}$ that converges to
  $y^1$.  Then $\sup_n z^n_2 < b_2$ and there is $w\in G$ such that
  $\sup_n z^n_2 \leq w_2 < b_2$. In view of Lemma~\ref{lem:21},
  $u^n = D^c\phi(z^n)$ is the only element of $\Arg(z^n)$. If
  $x^0\in \Arg(y^0)$, then $y^1$ stays strictly above the line segment
  $L(x^0,y^0)$ and, as $z^n\to y^1$, we can assume that same property
  holds for $(z^n)$.  By Lemmas~\ref{lem:22} and~\ref{lem:24}, the
  line segment $L(z^n,u^n)$ has negative slope and can intersect
  $L(y^0,x^0)$ only at $x^0$. It follows that $u^n$ belongs to the
  compact set $G(x^0,w)$. Continuity of $\phi=\phi_G$ from
  Lemma~\ref{lem:18} yields that any convergent subsequence of $(u^n)$
  goes to the unique $x^1 \in \dom{\Arg(y^1)}$.  Hence,
  $y^1 \in \dom{\nabla \phi}$ and $x^1 = D^c \phi(y^1)$, by the
  definition of $\nabla \phi$ on $\boundary{\dom{\phi}}$.

  Similar arguments show that if the ``corner'' point
  $\widehat y = (a_1,b_2)\in \udom{\Arg}$ and there are
  $z^0, z^1 \in\dom{\Arg}$ that belong to the interiors of different
  linear parts of $\boundary{\dom{\phi}}$, then
  $\widehat y \in \dom{\nabla \phi}$.
\end{proof}

\begin{proof}[Proof of Theorem~\ref{th:9}]
  From Lemmas~\ref{lem:21} and~\ref{lem:24} we deduce that
  \begin{displaymath}
    \dom{\nabla \phi} \setminus \udom{\Arg} \subset
    \boundary{\dom{\phi}} \cap G.
  \end{displaymath}
  Lemma~\ref{lem:17} shows that the boundary of $\dom{\phi}$ is
  contained in the union of two lines and that each of these lines is
  either vertical or horizontal. It follows that
  \begin{displaymath}
    \boundary{\dom{\phi}} \cap G =
    \boundary{\dom{\phi}} \cap E
  \end{displaymath}
  and we obtain~\eqref{eq:34}. Lemma~\ref{lem:25} states the structure
  of $\udom{\Arg}/\dom{\nabla \phi}$. Let
  $y\in \dom{\nabla \phi} \cap \udom{\Arg}$. Accounting
  for~\eqref{eq:33} we deduce that
  \begin{displaymath}
    y \not\in \boundary{\dom{\phi}} \cap E = \boundary{\dom{\phi}} \cap
    G. 
  \end{displaymath}
  Lemmas~\ref{lem:21} and~\ref{lem:24} now yield that $D^c\phi(y)$ is
  the only element of $\Arg(y)$.  Finally, identity~\eqref{eq:35}
  holds by the definition of $D^c\phi$.
\end{proof}

We recall that $\mathcal{D}$ denotes the family of graphs of strictly
decreasing functions $h=h(t)$ defined on closed intervals of $\R$ such
that $h$ and its inverse $h^{-1}$ are Lipschitz functions. We allow
for a degenerate case where the domain of $h$ is just a point. Thus,
$\R^2\subset \mathcal{D}$.

\begin{Theorem}
  \label{th:10}
  The exception set
  \begin{equation}
    \label{eq:36}
    \dom{\Arg}\setminus \left(\dom{\nabla \phi} \cap
      \udom{\Arg}\right) = D\cup E,  
  \end{equation}
  where $D$ is a countable union of sets in $\mathcal{D}$ and $E=E^G$
  is the union of horizontal and vertical line segments of $G$.
\end{Theorem}

We divide the proof into lemmas.  For $y\in \dom{\Arg}\setminus G$ and
the points $r\leq s$ in $\Arg(y)$, we denote by $\Delta(y,r,s)$ the
closed curved triangle bounded by the line segments $L(r,y)$,
$L(y,s)$, and the segment $G(r,s)$ of $G$; see Figure~\ref{fig:2}. If
$r=s$, then $\Delta(y,r,s) = L(r,y)=L(s,y)$; otherwise
$\interior{\Delta(y,r,s)}\not=\emptyset$.

\begin{Lemma}
  \label{lem:26}
  Let $y^0,y^1$ be distinct points in $\dom{\Arg}\setminus G$, let
  $r^i\leq s^i$ be in $\Arg(y^i)$, and denote
  $\Delta^i \set \Delta(y^i,r^i,s^i)$, $i=0,1$.
  \begin{enumerate}[label = {\rm (\alph{*})}, ref={\rm (\alph{*})}]
  \item \label{item:6} If $y^0\in \Delta^1$, then
    $\Delta^0\subset \Delta^1$.
  \item \label{item:7} If $y^0\not\in \Delta^1$ and
    $y^1\not \in \Delta^0$, then the intersection of ${\Delta^0}$ and
    ${\Delta^1}$ is at most one point, which is then either $r^1=s^0$
    or $s^1=r^0$.
  \end{enumerate}
\end{Lemma}

\begin{proof}
  If either~\ref{item:6} or \ref{item:7} fails to hold, then there are
  line segments $L^i \in \braces{L(r^i,y^i), L(s^i,y^i)}$, $i=0,1$,
  that intersect only at an interior point. We obtain a contradiction
  with Lemma~\ref{lem:23}.
\end{proof}

Lemma~\ref{lem:26}~\ref{item:6} yields a partial order relation on
$\dom{\Arg}\setminus G$: $y^0\prec y^1$ if
$y^0\in \Delta(y^1, r^1,s^1)$ for some $r^1\leq s^1$ in $\Arg(y^1)$.

\begin{Lemma}
  \label{lem:27}
  If $y^0,y^1$ belong to $\dom{\Arg}\setminus G$ and $y^0\prec y^1$,
  then
  \begin{displaymath}
    D(y^0,y^1) \set \descr{y\in \dom{\Arg}\setminus G}{y^0\prec y\prec
      y^1} \in \mathcal{D},   
  \end{displaymath}
  that is, $D(y^0,y^1)$ is the graph of a strictly decreasing function
  $h=h(t)$ on $[y^0_1,y^1_1]$ such that $h$ and its inverse $h^{-1}$
  are Lipschitz functions.
\end{Lemma}

\begin{figure}
  \centering
  \begin{tikzpicture}[scale = 4/10]
   
    \filldraw [black] (0,10) circle [radius=2pt] (10,0) circle
    [radius=2pt] (20,5) circle [radius=2pt] (10,7.5) circle
    [radius=2pt] (5,5) circle [radius=2pt];

    \draw (10,0) node [below right] {$r^1$} -- (0,10) node [left]
    {$y^1$} --(20,5) node [below right] {$s^1$};
    
    \filldraw [black] (10,5) circle [radius=2pt] (11,2) circle
    [radius=2pt] (13,4) circle [radius=2pt];

    \draw (11,2) node [below right] {$r^0$} -- (10,5) node [below
    left] {$y^0$} --(13,4) node [below right] {$s^0$};

    \draw[very thick,domain=9:10,smooth,variable=\x,black] plot
    ({\x},{-100/(\x) + 10});

    \draw[very thick,domain=20:22,smooth,variable=\x,black] plot
    ({\x},{-100/(\x) + 10}) node [above] {$G$};

    \draw[very thick,domain=11:13,smooth,variable=\x,black] plot
    ({\x},{-3/(\x-10) + 5});

    \draw[very thick,domain=10:11,smooth,variable=\x,black] plot
    ({\x},{(\x-10) + (\x-10)^2});

    \draw[very thick,domain=13:20,smooth,variable=\x,black] plot
    ({\x},{5 + 0.25*(\x-20) + 0.0578*(\x-20)^2 + 0.0061*(\x-20)^3});

    \draw[thin] (10,0) -- (10,7.5) node [above right] {$z^1$};
    \draw[thin] (20,5) -- (5,5) node [below left] {$z^0$};

    \draw[dashed, thin, gray] (10.2,0.24) -- (2, 8.733)
    --(18.6,4.747); \draw[dashed, thin, gray] (10.4,0.56) -- (4,7.6)
    --(17.2,4.620); \draw[dashed, thin, gray] (10.6,0.96) --
    (5.5,6.838) --(15.8,4.520); \draw[dashed, thin, gray]
    (10.8,1.44)-- (7.5,5.938) --(14.4,4.347); \draw[dashed, thin,
    gray] (10.9,1.71) -- (9,5.35) (14.4,4.347); \draw (6,6.6) node
    [above right] {$D$};

    \draw[very thick, domain=0:10,smooth,variable=\x,black] plot
    ({\x},{10-0.5*\x + \x*(\x-10)/60});
  \end{tikzpicture}
  \caption{The curve $D=D(y^0,y^1)$ separates the parts of the
    $c$-gradient flow of $\phi$ ending on the segments $G(r^1,r^0)$
    and $G(s^0,s^1)$.}
  \label{fig:2}
\end{figure}
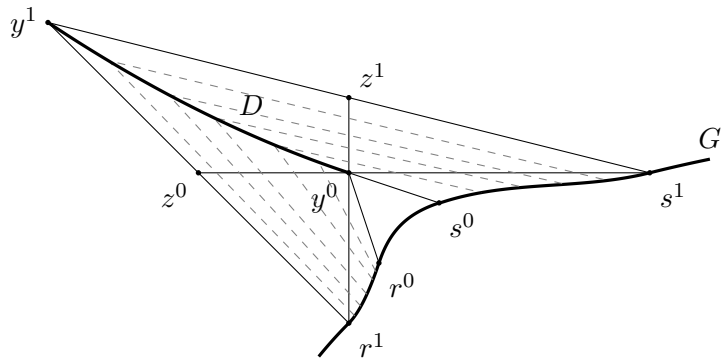

\begin{proof}
  We illustrate the proof on Figure~\ref{fig:2}. Without loss of
  generality we can assume that $y^0$ and $y^1$ are distinct points
  that stay above $G$. Let $r^i\leq s^i$ be in $\Arg(y^i)$. We have
  that $r^1\leq r^0\leq s^0\leq s^1$. If $y^0$ belongs to the line
  segment $L(y^1,s^1)$, then Lemma~\ref{lem:22} yields that
  \begin{displaymath}
    s^1 = r^0=s^0=D^c\phi(y^0)
  \end{displaymath}
  and then that $D = L(y^0,y^1)$.  Same lemma shows that the line
  segment $L(y^0,y^1)$ has a negative slope and thus, belongs to
  $\mathcal{D}$. The case, where $y^0\in L(y^1,r^1)$ is identical.

  Hereafter, we assume that $y^0\not\in L(y^1,s^1) \cup L(y^1,r^1)$
  or, equivalently, that $y^0\in \interior{\Delta(y^1,r^1,s^1)}$.
  Being a chord of the concave hyperbola
  \begin{displaymath}
    H^1 = \descr{z\in \R^2}{c(z,y^1) = \phi(y^1), \quad z_1>y_1},
  \end{displaymath}
  which touches $G$ from below, the line segment $L(r^1,s^1)$ stays
  below $G$.  It follows that $y^0$ belongs to the interior of the
  triangle with vertices $\braces{r^1,y^1,s^1}$. Hence, there are
  unique $z^1\in \ri{L(y^1,s^1)}$ and $z^0\in \ri{L(r^1,y^1)}$ such
  that the line segments $L(r^1,z^1)$ and $L(s^1,z^0)$ intersect at
  $y^0$.
  
  We observe that the convex polygon $P$ with the vertices
  $\braces{z^1,y^1,z^0,y^0}$ contains every
  $y\in \dom{\Arg} \setminus G$ such that $y^0\prec y\prec y^1$ and
  thus, contains $D$.  Being convex, $\psi$ is bounded on $P$. Hence,
  $\phi$ is bounded on $P$ as well. Moreover, as $P$ stays away from
  $G=\descr{x\in \dom{\phi}}{\phi(x) = 0}$, same boundedness property
  holds for $1/\phi$. If $y\in P\cap \dom{\nabla \phi}$, then
  Lemmas~\ref{lem:23} and~\ref{lem:24} show that $D^c\phi(y)$ belongs
  to the union of $G(r^1,r^0)$ and $G(s^0,s^1)$. In particular,
  $D^c\phi$ and then also $\nabla \phi$ are bounded on
  $P\cap \dom{\nabla \phi}$.  From Lemma~\ref{lem:22} we deduce the
  existence of negative constants $a$ and $b$ such that
  \begin{displaymath}
    -\infty < a\leq \frac{y_2-x_2}{y_1-x_1} \leq b < 0, \quad y\in P,
    \; x\in \Arg(y). 
  \end{displaymath}

  Let $y,z\in D$ be distinct. Lemma~\ref{lem:26} yields that either
  $y\prec z$ or $z\prec y$. Assuming that $ z\prec y$ we deduce the
  existence of $r,s \in \Arg(y)$ such that $r\leq s$ and
  $z\in \Delta(y,r,s)$.  The slope of $L(y,z)$ is then bounded from
  below by the slope of $L(y,r)$ and from above by the slope of
  $L(y,s)$, and thus is bounded in between by $a$ and $b$:
  \begin{displaymath}
    -\infty < a\leq \frac{y_2-z_2}{y_1-z_1} \leq b < 0. 
  \end{displaymath}
  Hence, the set $D$ has the required Lipschitz properties.
   
  It remains to be shown that the set $D$ is connected or,
  equivalently, that for every pair of distinct points $w^0\prec w^1$
  in $D$ there is $w\in D$, which is different from $w^0$ and $w^1$
  and such that $w^0\prec w \prec w^1$. Without loss of generality we
  can take $w^0=y^0$ and $w^1 = y^1$. We shall find the required $w$
  in $L(z^0,z^1)$.

  Let $z(t) = (1-t)z^0 + tz^1$, $t\in [0,1]$. From the
  non-intersection property of Lemma~\ref{lem:23} and the continuity
  of $\phi$ on its domain, we deduce that
  \begin{enumerate}
  \item If $t\in (0,1)$ and
    $\Arg(z(t)) \cap G(r^1,r^0)\not=\emptyset$, then
    $\Arg(z(s))\subset G(r^1,r^0)$, $0\leq s<t$.
  \item If $(t_n)\in [0,1]$ is such that $t_n \to t$ and
    $\Arg(z(t_n))\cap G(r^1,r^0)\not=\emptyset$, $n\geq 1$, then
    $\Arg(z(t)) \cap G(r^1,r^0)\not=\emptyset$.
  \end{enumerate}
  Similar properties (with obvious modifications in the first item)
  hold when $G(r^1,r^0)$ is replaced with $G(s^0,s^1)$. These
  properties readily yield the unique $t^* \in (0,1)$ such that
  $\Arg(z(t^*))$ intersects with \emph{both} $G(r^1,r^0)$ and
  $G(s^0,s^1)$. Clearly, $w = z(t^*)$ is different from both $y^0$ and
  $y^1$ and $y^0\prec w \prec y^1$, thanks to Lemma~\ref{lem:26}.
\end{proof}

\begin{Lemma}
  \label{lem:28}
  The set
  \begin{equation}
    \label{eq:37}
    D\set \descr{y\in \dom{\Arg}\setminus G}{\Arg(y) \;
      \text{contains at least 2 points}},
  \end{equation}
  if not empty, is a countable union of sets in $\mathcal{D}$. More
  precisely,
  \begin{displaymath}
    D = \cup_{n\geq 1} D(u^n,v^n) = \cup_{n\geq 1} \descr{y\in D}{u^n \prec y
      \prec v_n}, 
  \end{displaymath}
  for some $u^n \leq v^n$ in $D$, $n\geq 1$.
\end{Lemma}
\begin{proof}
  Clearly, $D = \cup_{n\geq 1} D(\frac1n)$, where
  \begin{displaymath}
    D(\epsilon) \set \descr{y\in D}{\abs{x^0-x^1}\geq \epsilon \text{
        for some } x^0,x^1\in \Arg(y)}, \quad \epsilon>0. 
  \end{displaymath}

  Let $\epsilon>0$.  We denote by $\widehat D(\epsilon)$ the set of
  \emph{minimal} elements of $D(\epsilon)$ with respect to the order
  relation $\prec$. In other words,
  $\widehat y\in \widehat D(\epsilon)$ if any $y\in D(\epsilon)$ such
  that $y\prec \widehat y$ coincides with $\widehat{y}$. From
  Lemma~\ref{lem:26} we deduce that $\widehat D(\epsilon)$ is
  countable.  Let $y\in D(\epsilon)$. If $y$ is not a minimal element,
  then there is $y'\in D(\epsilon)$ such that $y'\prec y$, $y\not=y'$.
  Being contained in $\Delta(y,u,v)$ for some $u\leq v$ in
  $\Arg{(y)}$, the set $\descr{z\in D(\epsilon)}{z\prec y'}$ is
  bounded. By the continuity of $\phi= \phi_G$, this set is closed and
  hence, contains some $\widehat y \in \widehat D(\epsilon)$.  It
  follows that
  \begin{displaymath}
    D(\epsilon)= \cup_{\widehat y \in \widehat D(\epsilon)} \descr{y\in
      D}{\widehat y \prec y}.
  \end{displaymath}
  Finally, for $y\in D$, Lemmas~\ref{lem:26} and~\ref{lem:27} show
  that $\descr{z\in D}{y\prec z}$ is the graph of a strictly
  decreasing function $h$ such that $h$ and $h^{-1}$ are locally
  Lipschitz. The result readily follows.
\end{proof}

\begin{proof}[Proof of Theorem~\ref{th:10}]
  By Theorem~\ref{th:9} representation~\eqref{eq:36} holds if we add
  to the set $D$ given by~\eqref{eq:37} at most 2 points.
  Lemma~\ref{lem:28} yields the result.
\end{proof}

\begin{Lemma}
  \label{lem:29}
  Let $D$ be given by~\eqref{eq:37} and
  \begin{align*}
    S = \descr{y \in \dom{\Arg}}{ \Arg(y) \;\text{contains at least 3
    points}}. 
  \end{align*}
  Then $S$ is countable and there are Borel functions
  $\map{g_i}{D}{G}$, $i=1,2$, such that
  \begin{displaymath}
    \Arg(y) = \braces{g_1(y),g_2(y)}, \; g_1(y)\not = g_2(y), \quad
    y\in D \setminus S. 
  \end{displaymath}  
\end{Lemma}

\begin{proof}
  In view of Lemma~\ref{lem:28}, it is sufficient to prove the result
  for the sets $D' = D(u,v)$ and $S' = S\cap D'$, where $u\prec v$ in
  $D$. Let $r\leq s$ be distinct elements of $\Arg(u)$. The functions
  \begin{align*}
    g_1(y) &= \max\descr{x\in \Arg(y)}{x\leq r},  \\
    g_2(y) &= \min\descr{x\in \Arg(y)}{x\geq s}, 
  \end{align*}
  map $D'$ to $G$ and are monotone with respect to the order relations
  $\prec$ on $D'$ and $\leq$ on $G$. Thus, their respective sets
  $(R_i)$ of discontinuities are countable. From Lemma~\ref{lem:26} we
  deduce that $S'\subset R_1 \cup R_2$ and from the continuity of
  $\phi$ that $g_i(y) \in \Arg(y)$, $y\in D'$. The proof readily
  follows.
\end{proof}

\bibliographystyle{plainnat} \bibliography{../../bib/finance}

\end{document}